\documentclass[oneside,final,11pt]{amsart}
\usepackage[T1]{fontenc}
\usepackage[latin9]{inputenc}
\usepackage{geometry}
\usepackage{graphicx}
\usepackage{soul}
\usepackage{xcolor}
\usepackage{amssymb}

\newtheorem{theorem}{Theorem}[section]
\newtheorem{lemma}[theorem]{Lemma}
\newtheorem{proposition}[theorem]{Proposition}
\newtheorem{corollary}[theorem]{Corollary}
{ \theoremstyle{definition}
\newtheorem{definition}[theorem]{Definition}}
{ \theoremstyle{definition}
}
{ \theoremstyle{remark}
\newtheorem{remark}[theorem]{Remark}}
\usepackage{placeins}
\usepackage{cite}
\usepackage{caption}
\usepackage{enumerate}
\usepackage{afterpage}
\usepackage{enumitem}
\usepackage{bmpsize}
\usepackage{amsmath}

\DeclareMathOperator{\arccot}{arccot}

\title{Interacting particle systems at the edge of multilevel Jack processes}
\date{\today}
\author{Evgeni Dimitrov and Panagiotis Lolas}

\begin{document}

\maketitle 

\begin{abstract}
We consider a multilevel continuous time Markov chain $X(s;N) = (X_i^j(s;N): 1 \leq i \leq j \leq N)$, which is defined by means of Jack symmetric functions and forms a certain discretization of the multilevel Dyson Brownian motion. The process $X(s;N)$ describes the evolution of a discrete interlacing particle system with push-block interactions between the particles, which preserve the interlacing property. 
We study the joint asymptotic separation of the particles at the right edge of the ensemble as the number of levels and time tend to infinity and show that the limit is described by a certain zero range process with local interactions.
\end{abstract}

\tableofcontents
\section{Introduction and main results}\label{Section1}
The main results of this paper are contained in Section \ref{Section1.2}. The section below gives background for the main object we study, which is a certain interacting particle system with push-block dynamics.

\subsection{Preface}\label{Section1.1}
During the last two decades there has been significant progress in understanding the long time nonequilibrium behavior of interacting particle systems and random growth models that belong to the so-called KPZ universality class. An important role for this success, has been played by {\em integrable} (or exactly solvable) models. Integrability in this case refers to the fact that these systems typically come with some enhanced algebraic structure, which makes them more amenable to detailed analysis and hence provides the most complete access to various phenomena such as phase transition, intermittency, scaling exponents, and fluctuation statistics.

One particular algebraic framework, which has enjoyed substantial interest and success in analyzing various probabilistic systems in the last several years, is the theory of Macdonald processes \cite{BorCor}. Macdonald processes are defined in terms of a remarkable class of symmetric polynomials, called Macdonald symmetric polynomials, which are parametrized by two numbers $(q,t)$ - see \cite{Mac}. By leveraging some of their algebraic properties, Macdonald processes have proved useful in solving a number of problems in probability theory, including computing exact Fredholm determinant formulas and associated asymptotics for one-point marginal distributions of the O'Connel-Yor semi-discrete directed polymer \cite{BorCor, BCF}; log-gamma discrete directed polymer \cite{BorCor, BCR};  KPZ/stochastic heat equation \cite{BCF}; $q$-TASEP \cite{Bar15, BorCor, BorCor2, BCS} and $q$-PushASEP \cite{BP, CP}.

There is a rich class of integrable models for interacting particle systems that comes from the multivariate continuous time Markov chains, which preserve Macdonald processes. These dynamics are called {\em push-block} in \cite{BP}, but we will refer to them as {\em multilevel Macdonald processes} or {\em MMPs}. MMPs describe certain interacting particle systems with global interactions, whose state space is given by interlacing particle configurations with integer coordinates. For a definition of MMPs we refer the reader to Section 2.3.3 of \cite{BorCor}; however, we remark that the construction there is parallel to those of \cite{Bor11,BorFer} and is based on a much earlier idea of \cite{DF} (see also \cite{BP} for a more general discussion). 

Two particular cases of the MMPs, which have been studied extensively, are $t = 0$ (this degenerates Macdonald to $q$-Whittaker symmetric functions) and $q = t$ (this degenerates Macdonald to Schur symmetric functions). One reason these two cases have received much attention is because of their connection to the KPZ equation and universality class (see \cite{BorCor, BCF, BorFer} and the references therein). Another, more technical, reason is that these two cases come with a certain algebraic structure, which can be exploited to obtain concise formulas for a large class of observables. Specifically, in the Schur case ($q=t$) the dynamics is described by a determinantal point process, whose correlation kernel has a relatively simple form along {\em space-like} paths \cite{BorFer}. In the $q$-Whittaker case ($t= 0$) and also for generic $(q,t)$ parameters the algebraic tools that provide access to detailed asymptotic analysis are the Macdonald difference operators \cite{BorCor}.\\

In this paper we study a different case of the MMP, when $t = q^\theta$ and $q \rightarrow 1$, where $\theta > 0$. This parameter specialization degenerates the Macdonald to the Jack symmetric functions and we call the resulting dynamics {\em multilevel Jack processes} or {\em MJPs}. MJPs form a one-parameter generalization of the multilevel Schur dynamics and they degenerate to the latter when $\theta = 1$. One reason that MJPs have received relatively little attention is because the existing methods for the Schur and $q$-Whittaker case are not directly applicable to this setting. In particular, for $\theta \neq 1$ MJPs lose the determinantal point process structure of the Schur dynamics, and the $q$-moments method that comes from the Macdonald difference operators fails to produce useful formulas for observables.

One motivation for studying MJPs comes from their connections with random matrix theory. In \cite{GS1} it was shown that under a diffuse scaling limit the MJPs converge to a  simple diffusion process that depends on a parameter $\beta = 2\theta$ and is called {\em multilevel Dyson Brownian motion} ({\em MDBM}). This process generalizes the interlacing reflected Brownian motions process of Warren \cite{War}, which is recovered when $\beta = 2$. In addition, when projected on the top row the MDBM agrees with the Dyson Brownian motion and its fixed time distribution is given by the Hermite $\beta$ corners process. Another important feature of MJPs is that their fixed time distribution of the top level is described by the {\em discrete $\beta$-ensemble} of \cite{BGG}. The discrete $\beta$-ensembles are probability distributions on particle ensembles, which are discretizations for the general-$\beta$ log-gases of random matrix theory. The link between MJPs and the discrete $\beta$-ensemble is described in Section \ref{Section5} below and it plays an important role in our arguments. 

In view of its connection to the MDBM and the discrete $\beta$-ensembles, but also as an interesting integrable model in its own right, it is desirable to develop tools and analyze the MJP and this is the main purpose of this paper. Our main results (Theorems \ref{Theorem1} and \ref{Theorem2} below) describe the asymptotic distribution of the separation of the particles at the right edge of a particular MJP as the number of levels and time go to infinity with the same rate. In this limit we show that the dynamics of the gaps between particles converge to an explicit stationary continuous time Markov chain. Interestingly, in the limit the interactions of the particles on the right edge with the rest of the diagram disappear. I.e. the particles on the right edge decouple from the others and their limiting evolution is based on {\em local} interactions among themselves. We remark that the latter phenomenon was observed in the case of MDBM in \cite{GS2}, where analogous (continuous) versions of our results were obtained. 

Our methods are largely influenced by \cite{GS2}; however, we emphasize that we make substantial modifications to their arguments. As basic ingredients for our proofs we use results available for the discrete $\beta$-ensemble such as the law of large numbers for the empirical measures and the large deviation estimates for the right-most particle \cite{BGG}. These substitute asymptotic results from random matrix theory that were utilized in \cite{GS2}. In addition, due to the discrete nature of our process, we achieve various significant simplifications of our proofs, especially for the dynamical setting. So for example, we completely avoid using strong results from SDE theory such as rigidity estimates for Brownian motion, and instead rely on more direct probabilistic arguments.

We now turn to formulating our problem and presenting our results in detail.

\subsection{The process $X(s;N)$}\label{Section1.2}
We start by describing the main object that we study, which is a certain $N(N+1)/2$-dimensional process that we denote by $X(s;N) = \left(X^j_i(s;N) : \hspace{1mm} 1 \leq i \leq j \leq N\right),s \geq 0$. The state space of this process is the space of Gelfand-Tsetlin patterns $\mathbb{GT}^N$, defined by
\begin{equation}\label{S1GT}
\mathbb{GT}^N = \{ y = (y_i^j )_{1 \leq i \leq j \leq N} \in \mathbb{Z}^{N(N+1)/2}: \hspace{1mm} y_i^{j+1} \geq y_i^j \geq y_{i+1}^{j+1}, \hspace{1mm} 1 \leq i \leq j \leq N-1\}
\end{equation}
At time $0$ we assume that the process starts from $X(0;N) = 0^{N(N+1)/2}$. In what follows we describe the evolution of the particles, and to make illustrations clearer we will work with the deterministically transformed process $x_i^j = X_i^j - i + 1$. We interpret the coordinates $x^j_i$ as positions of particles, and we also use $x_i^j$ to label them.  The initial configuration for the transformed process is given in Figure \ref{S1_1} and the dynamics is as follows.

Each of the coordinates (particles) $x^j_i$ has its own exponential clock with rate given by
\begin{equation}\label{S1Rates}
\begin{split}
q(x_i^j) =  \hspace{1mm} &\theta\cdot \prod_{r = 1}^{i-1} \frac{x_r^j - x_i^j + (\theta - 1) (i-r+1)}{x_r^j - x_i^j + (\theta - 1)(i-r+1) + \theta} \cdot  \frac{x_r^{j-1} - x_i^j + (\theta - 1) (i-r) -\theta }{x_r^{j-1} - x_i^j + (\theta - 1)(i-r)  - 1} \\
&\times \prod_{n = i}^{j-1}\frac{x_i^j - x_{n+1}^j + (\theta - 1) (n-i)}{x_i^j - x_{n+1}^j + (\theta - 1 )(n-i + 1) } \cdot  \frac{x_i^{j} - x_n^{j-1} + (\theta - 1) (n-i) +\theta }{x_i^{j} - x_n^{j-1} + (\theta - 1)(n-i)  + 1},
\end{split}
\end{equation}
where $\theta > 0$ is fixed throughout this discussion. This particular form of the jump rates is a consequencence of our definition of the dynamics through Jack polynomials - see Section \ref{Section2} below. Although the expression in (\ref{S1Rates}) is rather involved, it turns out that it provides the correct way to discretize the dynamics of the multilevel Dyson Brownian motion \cite{GS1}.

All clocks are assumed to be independent and when the $x_i^j$ clock rings the particle jumps to the right by $1$. We observe that the above jump rates induce the following push-block dynamics, which ensure that the process $x_i^j$ will always satisfy  $x_i^{j+1} \geq x_i^j > x_{i+1}^{j+1}$ for $1 \leq i \leq j \leq N-1$  (i.e. our original process $X(s;N),{s\geq 0}$ will never leave $\mathbb{GT}^N$).

\vspace{-10mm}

\begin{figure}[h]
\centering
\begin{minipage}{.40\textwidth}
\vspace{8.5mm}
  \centering
  \includegraphics[width=0.9\linewidth]{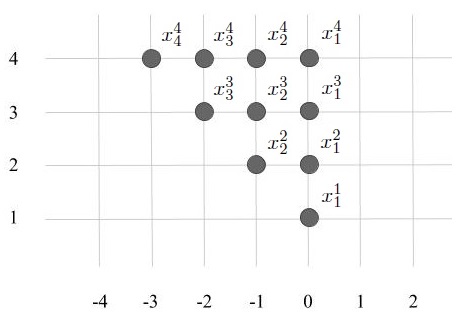}
\captionsetup{width=.9\linewidth}
  \caption{Initial condition for the process $x_i^j = X_i^j - i + 1$ when $N = 4$.}
\label{S1_1}	
\end{minipage}
\hspace{2mm}
\begin{minipage}{.50\textwidth}
\vspace{8.5mm}
  \centering
  \includegraphics[width=0.9\linewidth]{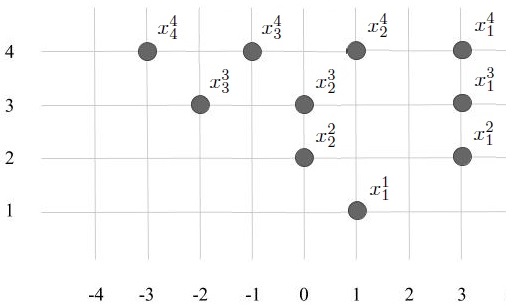}
\captionsetup{width=.9\linewidth}
  \caption{Sample particle configuration for $N = 4$.}
\label{S1_2}	
\end{minipage}
\end{figure}

From (\ref{S1Rates}) we see that the jump rate of a particle $x_i^j$ depends only on the positions of the particles on rows $j$ and $j-1$. If $i > 1$ and $x_{i-1}^{j-1} = x^j_i + 1$ , then we notice that the first product in (\ref{S1Rates}) vanishes and so $q(x^j_i) = 0$. We say that the particle $x^{j-1}_{i-1}$ has blocked $x^j_i$ and the latter cannot jump to the right. In Figure \ref{S1_2} particle $x^4_4$ is blocked by its bottom right neighbor $x_3^3$ and particle $x_2^2$ by its bottom right neighbor $x_1^1$. 

We next suppose that we have $x_i^{j+1} = x_i^j$ and that $x_i^j$ has jumped to the right by $1$. In this case we see that the denominator of $q(x^{j+1}_i)$ in (\ref{S1Rates}) vanishes and so the jump rate becomes infinite. This causes $x_i^{j+1}$ to immediately jump to the right together with $x_i^j$ and we say that $x_i^j$ has pushed $x_i^{j+1}$ to the right.  This pushing mechanism continues upward, so if for example the move $x_i^{j+1} \rightarrow x_i^{j+1} + 1$ has made this particle surpass $x_i^{j+2}$ then $x_i^{j+2}$ is also pushed to the right by $1$. In general, if a particle $x_i^j$ has jumped to the right by $1$, then we need to find the longest string of particles such that $x_i^j = x_i^{j+1} = \cdots = x_i^{j + r}$ and move all of them simultaneously to the right by $1$. 

We illustrate the latter push dynamics with an example. If $x_2^3$ jumps to the right twice and then $x_1^2$ jumps to the right once, we will obtain Figure \ref{S1_3} from Figure \ref{S1_2}. The first jump of $x_2^3$ simply moves that particle to the right by $1$. The second jump moves it to the right, but also pushes $x_2^4$ to the right by $1$. Finally, when $x_1^2$ moves to the right it pushes $x_1^3$, which in turn pushes $x_1^4$ to the right and so altogether all three particles move to the right by $1$.
\begin{figure}[h]
\centering
  \scalebox{0.7}{\includegraphics[width=0.9\linewidth]{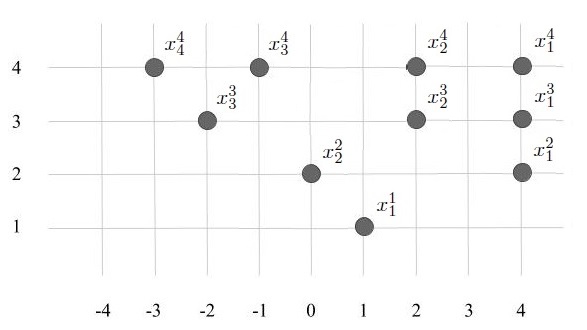}}
\captionsetup{width=.95\linewidth}
  \caption{Result of particle $x_2^3$ jumping twice and $x_1^2$ jumping once, starting from Figure \ref{S1_2}}
  \label{S1_3}
\end{figure}
A simple heuristic to help the reader remember how the push-block dynamics works is that lower particles are heavier and higher particles are lighter. Then when a heavy particle moves it pushes all lighter particles above it, and when a lighter particle tries to jump and there is a heavier one blocking it, it will not move.\\

The above push-block dynamics ensures that our process $X(s;N), s \geq 0$ never leaves $\mathbb{GT}^N$ and is thus a well-defined process there. Although the dynamics that we presented above is certainly sufficient to define the process $X(s;N), s \geq 0$ we will postpone a formal definition until Definiton \ref{mainDef} later in the text. That definition will be based on the formalism of multilevel probability distributions and stochastic dynamics, built from Jack polynomials, which is presented in Section \ref{Section2} below.

In what follows we summarize the main results of our paper for the process $X(s;N), s\geq 0$.

\begin{theorem}\label{Theorem1}
Let $X(s;N), s\geq 0$ be as in Definition \ref{mainDef} with $\theta \geq 1$. Fix $t > 0$, $s \geq 0$ and $k \in \mathbb{N}$. Then as $N \rightarrow \infty$ the sequence
$$\left(X_1^N\left (tN + s;N\right)-X_{1}^{N-1}\left(tN + s;N\right),...,X_{1}^{N-k+1}\left(tN + s;N\right)-X_{1}^{N-k}\left(tN + s;N\right)\right)$$ converges in law to a random vector $(Q_1,...,Q_k) \in \mathbb{Z}_{\geq 0}^k$, where $Q_1,...,Q_k$ are i.i.d. random variables with
$$\mathbb{P}( Q_1 = n) = (1- p)^{-\theta}\frac{\Gamma(n+\theta)}{\Gamma(n+1)\Gamma(\theta)} p^n, \hspace{1mm} n \in \mathbb{Z}_{\geq 0} \mbox{, and } p = \frac{\sqrt{t}}{1 + \sqrt{t}}.$$
\end{theorem}
\begin{remark}
In \cite{GS2}, the authors considered the same limit as in Theorem \ref{Theorem1} for $t = \theta^{-1}$, for the multilevel Dyson Brownian motion. In the limit they also obtained that the separations of adjacent particles on the right edge at fixed time are i.i.d. random variables, but with the Gamma distribution with density 
$$f(x) = \frac{ \theta^\theta}{\Gamma(\theta)} x^{\theta -1} e^{-\theta x}.$$
 In this sense, we see that Theorem \ref{Theorem1} produces a discrete version of the result in \cite{GS2}.
\end{remark}

Our next aim is to formulate a dynamic multilevel convergence result about the process $X(s;N),s\geq 0$, but before we do we describe the limiting object, which is a certain zero range process with local interactions.

Let us fix $k\geq 1$, $t > 0$ and $\theta > 0$. Suppose we have $k$ piles of particles at locations $1,...,k$, and at time $s$ the $i$-th pile contains a non-negative integer number of particles $Q_i(s)$. In addition, we assume we have a pile with infinitely many particles at location $0$ and a sink at location $k+1$. The $i$-th pile with $i \in \{1,...,k\}$ has an exponential clock with parameter $\lambda_i(s) = \theta\cdot \frac{\theta + Q_i(s)}{1 + Q_i(s)}$. The clocks are independent of each other and when the $i$-th clock rings, a particle from the closest non-empty pile to the left jumps into pile $i$. The infinite pile at location $0$, ensures that there is always a non-empty pile to the left and is a source for new particles to enter the system. In addition, the sink has an exponential clock with constant parameter $\lambda_{sink} = \theta \cdot \frac{1 + \sqrt{t}}{\sqrt{t}}$ and when the clock rings a single particle jumps from the nearest non-empty pile to the left into the sink and disappears. A graphical representation of this process is given in Figure \ref{S1_4}. We also isolate this construction in a definition for future reference.
\begin{figure}[h]
\centering
  \scalebox{0.8}{\includegraphics[width=0.9\linewidth]{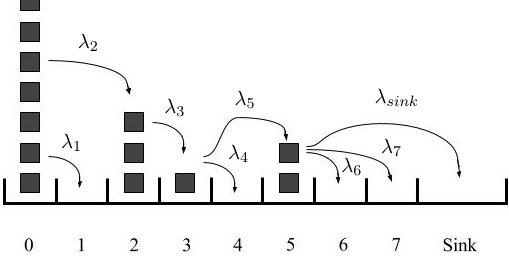}}
\captionsetup{width=.95\linewidth}
  \caption{The process $Q(s)$ when $k = 7$ at a fixed time. The arrows indicate the possible jumps that particles can make and the rates with which they happen are written above them }
  \label{S1_4}
\end{figure}

\begin{definition}\label{S1DefQ}
For $k\geq 1$, $t > 0$ and $\theta > 0$ we let $Q(s) = (Q_1(s),...,Q_k(s))$ be the continuous time Markov chain on $\mathbb{N}_0^k$ defined through the dynamics in the previous paragraph and with initial distribution such that $Q_1(0),...,Q_k(0)$ are i.i.d. random variables with
$$\mathbb{P}( Q_1(0) = n) = (1- p)^{-\theta}\frac{\Gamma(n+\theta)}{\Gamma(n+1)\Gamma(\theta)} p^n, \hspace{1mm} n \in \mathbb{Z}_{\geq 0} \mbox{, and } p = \frac{\sqrt{t}}{1 + \sqrt{t}}.$$
\end{definition}
It is easy to check that $Q(s)$ is a stationary pure jump continuous time Markov process and we view it as an element in $D^k$ - the space of right continuous left limited functions from $[0,\infty)$ to $\mathbb{N}_{\geq 0}^k$ with the usual Skorohod topology (see e.g. \cite{EK}). With this notation we formulate the following theorem.
\begin{theorem}\label{Theorem2}
Let $X(s;N), s\geq 0$ be as in Definition \ref{mainDef} with $\theta \geq 1$. Fix $t > 0$ and $k \in \mathbb{N}$. Then 
$$\left(X^N_1(tN + s;N) - X_1^{N-1}(tN + s;N),...,X^{N-k+1}_1(tN + s;N) - X_1^{N-k}(tN+ s;N)\right),  s \geq 0$$
 converges in the limit $N \rightarrow \infty$ in law on $D^k$ to the process $Q(s)$ from Definition \ref{S1DefQ}.
\end{theorem}
\begin{remark}
Theorem \ref{Theorem2} resembles Theorem 1.6 in \cite{GS2}, where the authors consider the same limit with $t = \theta^{-1}$ for multilevel Dyson Brownian motion. In that setting, the limiting object is a certain stationary Markov process, which the authors define as a weak solution of a certain system of SDEs that have a {\em local form}.
\end{remark}

\subsection{Outline and acknowledgements}\label{Section1.3}
In Section \ref{Section2} we provide the necessary background on how to develop stochastic dynamics from Jack polynomials and their positive specializations. In Section \ref{Section3} we use that the fixed time distribution of the top row of $X(s;N), s\geq 0$ is described by the discrete $\beta$-ensemble of \cite{BGG}. Relying on various previously known results about the discrete $\beta$-ensemble, such as law of large numbers of the empirical measures and large deviation estimates for the edge, we prove Theorem \ref{Theorem1}. In Section \ref{Section4} we prove Theorem \ref{Theorem2} using Martingale Problem convergence techniques, in the spirit of Stroock and Varadhan. Section \ref{Section5} explains the link between $X(s;N), s\geq 0$ and the discrete $\beta$-ensemble and supplies the proofs of various results used throughout the text.\\

The authors would like to thank Vadim Gorin for suggesting this problem to them and for numerous fruitful discussions.

\section{Multilevel dynamics via Jack polynomials} \label{Section2} 
The process $X(s;N)$ from Section \ref{Section1.2} is a special case of a multilevel dynamics, defined through Jack polynomials. This section provides the necessary background for the construction of these dynamics and forms the theoretical basis for our definition of $X(s;N)$, presented in the beginning of Section \ref{Section3}. 

\subsection{General definitions}\label{Section2.1}

We summarize some facts about partitions and Jack symmetric polynomials, using \cite{Mac} and Section 2.1 in \cite{GS1} as main references. Readers familiar with these polynomials can proceed to Section \ref{Section2.2}.

We start by fixing some terminology and notation. A {\em partition} is a sequence $\lambda =(\lambda_1, \lambda_2,\cdots)$ of non-negative integers such that $\lambda_1 \geq \lambda_2 \geq \cdots$ and all but finitely many elements are zero. We denote the set of all partitions by $\mathbb{Y}$. The {\em length} $\ell(\lambda)$ is the number of non-zero $\lambda_i$ and the {\em weight} is given by $|\lambda| = \lambda_1 + \lambda_2 + \cdots$. For $N \geq 0$ we let $\mathbb{Y}^N$ be the set of partitions of length at most $N$, where we agree that $\mathbb{Y}^0$ consists of a single partition of weight $0$, which we denote by $\varnothing$. We say that $\lambda, \mu \in \mathbb{Y}$ {\em interlace} and write $\mu \prec \lambda$ if 
$$\lambda_1 \geq \mu_1 \geq \lambda_2 \geq \mu_2 \geq \cdots.$$

A {\em Young diagram} is a graphical representation of a partition $\lambda$, with $\lambda_1$ left justified boxes in the top row, $\lambda_2$ in the second row and so on. In general, we do not distinguish between a partition $\lambda$ and the Young diagram representing it. The {\em conjugate} of a partition $\lambda$ is the partition $\lambda'$, whose Young diagram is the transpose of the diagram $\lambda$. In particular, we have the formula $\lambda_i'  =|\{j \in \mathbb{N}: \lambda_j \geq i\}|$. For a box $\square =(i,j)$ of a Young diagram $\lambda$ (i.e., a pair $(i,j)$ with $\lambda_i \geq j$) we let
$$\mbox{ $a(i,j;\lambda)=\lambda_i-j$,\hspace{5mm} $l(i,j;\lambda)=\lambda_j'-i,$\hspace{5mm}  $a'(i,j; \lambda)=j-1$,\hspace{5mm}  $l'(i,j; \lambda)=i-1$}.$$
The quantities $a(i,j;\lambda)$ and $l(i,j;\lambda)$ are called the {\em arm} and {\em leg lengths} respectively, while $a'(i,j;\lambda)$ and $l'(i,j;\lambda)$ are called the {\em co-arm} and {\em co-leg lengths} respectively. When $\lambda$ is clear from context we will omit it from the notation and write $a(i,j)$ (or $a(\square)$) and $l(i,j)$ (or $l(\square)$).\\

Let $\Lambda_X$ denote the $\mathbb{Z}_{\geq 0}$ graded algebra over $\mathbb{C}$ of symmetric polynomials in countably many variables $X = (x_1,x_2,...)$ of bounded degree, see e.g. Chapter I of \cite{Mac} for general information on $\Lambda_X$. One way to view $\Lambda_X$ is as an algebra of polynomials in Newton power sums
$$p_k(X) = \sum_{i = 1}^\infty x_i^k, \hspace{5mm} \mbox{for } k\geq 1.$$ 
For any partition $\lambda$ we define 
$$p_\lambda(X) = \prod_{i = 1}^{\ell(\lambda)} p_{\lambda_i}(X),$$
and note that $p_\lambda(X)$, $\lambda \in \mathbb{Y}$ form a linear basis of $\Lambda_X$.

In what follows we fix a parameter $\theta \in (0,\infty)$. Unless the dependence on $\theta$ is important we will suppress it from our notation, similarly for the variable set $X$. We write $J_\lambda(X;\theta)$ for the {\em Jack polynomial} with parameter $\theta$,  which is indexed by the partition $\lambda$. The polynomials $J_\lambda(X;\theta)$, $\lambda \in \mathbb{Y}$ form another linear basis of $\Lambda_X$ and many of their properties can be found in Section 10, Chapter VI of \cite{Mac}. We note that in \cite{Mac} Macdonald uses the parameter $\alpha$, corresponding to $\theta^{-1}$ in our notation. Our choice to work with $\theta$ is made after \cite{KOO}. If we specialize the variables $x_{N+1}, x_{N+2},...$ to all equal $0$ in the formula for $J_\lambda(X;\theta)$ we obtain a symmetric polynomial in $N$ variables $(x_1,...,x_N)$, denoted by $J_\lambda(x_1,...,x_N;\theta)$. The leading term of $J_\lambda(X;\theta)$ and $J_\lambda(x_1,...,x_N;\theta)$ is given by $x_1^{\lambda_1}\cdots x_{\ell(\lambda)}^{\lambda_{\ell(\lambda)}}$ (if $N \geq \ell(\lambda)$) and we have the following Sekiguchi differential operator eigenrelation
\begin{equation}
\begin{split}
&\frac{1}{\prod_{1\leq i<j \leq N}(x_i-x_j)}\det\left[x_i^{N-j}\left(x_i\frac{\partial}{\partial{x_i}}+(N-j)\theta+u\right)\right]J_{\lambda}(x_1,...,x_N;\theta)= \\
&\left(\prod_{i=1}^N\left(\lambda_i+(N-i)\theta+u\right)\right)\cdot J_{\lambda}(x_1,...,x_N;\theta).
\end{split}
\end{equation}
The latter two properties uniquely define $J_\lambda(x_1,...,x_N;\theta)$ and $J_\lambda(X;\theta)$. We also use the dual Jack polynomials $\widetilde{J}_{\lambda}$, which differ from $J_\lambda$ by an {\em explicit} constant, depending on $\lambda$:
\begin{equation}\label{S2eqDual}
\widetilde{J}_{\lambda}=J_{\lambda} \cdot \prod_{\square\in{\lambda}}\frac{a(\square)+\theta{l(\square)}+\theta}{a(\square)+\theta{l(\square)+1}}.
\end{equation}

We next proceed to define the {\em skew Jack polynomials} (see Chapter VI in \cite{Mac} for details). Take two sets of variables $X = (x_1,x_2,...)$ and $Y = (y_1,y_2,...)$ and a symmetric polynomial $f$ in countably many variables. Let $(X,Y)$ denote the union of sets of variables $X$ and $Y$. Then we can view $f(X,Y) \in \Lambda_{(X,Y)}$ as a symmetric polynomial in $x_i$ and $y_j$ together. More precisely, if 
$$f = \sum_{\lambda \in \mathbb{Y}}C_\lambda p_\lambda =  \sum_{\lambda \in \mathbb{Y}}C_\lambda \prod_{i = 1}^{\ell(\lambda)}p_{\lambda_i}$$
is the expansion of $f$ in the basis $p_\lambda$ (in the above sum $C_\lambda = 0$ for all but finitely many terms), then
$$f(X,Y) =  \sum_{\lambda \in \mathbb{Y}}C_\lambda \prod_{i = 1}^{\ell(\lambda)}(p_{\lambda_i}(X) + p_{\lambda_i}(Y)).$$ 
In particular, we see that $f(X,Y)$ is the sum of products of symmetric polynomials in $x_i$ and symmetric polynomials in $y_j$. The skew Jack polynomials $J_{\lambda/\mu}$ are defined as the coefficients in the expansion
\begin{equation}\label{skewJ}
J_{\lambda}(X,Y;\theta) = \sum_{\mu \in \mathbb{Y}} J_\mu(X;\theta) J_{\lambda/\mu}(Y;\theta).
\end{equation}
\begin{remark} The skew Jack polynomial $J_{\lambda/\mu}$ is $0$ unless $\mu \subset \lambda$ (i.e. $\lambda_i \geq \mu_i$ for $i \in \mathbb{N})$, in which case it is homogeneous of degree $|\lambda| - |\mu|$. When $\lambda = \mu$, $J_{\lambda/\mu} = 1$ and if $\mu = \varnothing$, then $J_{\lambda/\mu} = J_\lambda$.
\end{remark}
One similarly defines $\widetilde{J}_{\lambda/\mu}$ as the coefficients in the expansion
$$\widetilde{J}_{\lambda}(X,Y;\theta) = \sum_{\mu \in \mathbb{Y}} \widetilde{J}_\mu(X;\theta) \widetilde{J}_{\lambda/\mu}(Y;\theta).$$
We record the following generalization of (\ref{skewJ}) for later use (cf. Section 7 in Chapter VI of \cite{Mac}):
\begin{equation}\label{skewJ2}
\begin{split}
J_{\lambda/\kappa}(X,Y;\theta) = \sum_{\mu \in \mathbb{Y}} J_{\mu/\kappa}(X;\theta) J_{\lambda/\mu}(Y;\theta),\\
\widetilde{J}_{\lambda/\kappa}(X,Y;\theta) = \sum_{\mu \in \mathbb{Y}} \widetilde{J}_{\mu/\kappa}(X;\theta) \widetilde{J}_{\lambda/\mu}(Y;\theta),
\end{split}
\end{equation}
where $\kappa \in \mathbb{Y}$. Thus (\ref{skewJ}) is a special case of (\ref{skewJ2}) when $\kappa = \varnothing$.\\

An algebra homomorphism $\rho$ from $\Lambda$ to the set of complex numbers is called a {\em specialization}. If $\rho$ takes positive values on all (skew) Jack polynomials, it will be called {\em Jack-positive}. From \cite{KOO} we have the following classification of all Jack-positive specializations.
\begin{proposition}\label{propClass}
For any fixed $\theta>0$, Jack-positive specializations can be parametrized by triplets $(\alpha,\beta,\gamma)$, where $\alpha,\beta$ are sequences of real numbers with 
$$\alpha_1\geq{\alpha_2}\geq \cdots \geq{0}, \hspace{5mm} \beta_1\geq{\beta_2}\geq \cdots\geq{0}, \hspace{5mm}\sum_{i = 1}^\infty(\alpha_i+\beta_i)<\infty$$
and $\gamma$ is a non-negative real number. The specialization corresponding to a triplet $(\alpha,\beta,\gamma)$ is given by its values on the Newton power sum $p_k,k\geq{1}$:
$$p_1\mapsto p_1(\alpha,\beta,\gamma)=\gamma+\sum_{i= 1}^\infty(\alpha_i+\beta_i),$$
$$p_k\mapsto p_k(\alpha,\beta,\gamma)=\sum_{i = 1}^\infty{\alpha_i}^k+(-\theta)^{k-1}\sum_i{\beta_i}^k,k\geq{2}.$$
\end{proposition}
The specialization with all parameters equal to $0$ is called the {\em empty} specialization. It maps a polynomial to its constant term (i.e. the degree zero summand). 

Throughout this paper we will work with two specializations from Proposition \ref{propClass}. The first is denoted by $\mathfrak{a}^N$ and corresponds to taking $\alpha_1 = \cdots = \alpha_N = \mathfrak{a}$ and all other $\alpha, \beta$ and $\gamma$ parameters are set to zero. The second is the {\em Plancherel specialization} $\mathfrak{r}_s$, which satisfies $\gamma=s$ and all other parameters are set to $0$. 

We record some well-known explicit formulas for Jack-positive specializations (see e.g. Propositions 2.2, 2.3 and 2.4 in \cite{GS1}). In the following we write ${\bf 1}_E$ for the indicator function of the set $E$ and $(b)_n$ for the Pochhammer symbol $b(b+1)\cdots(b+n-1)$.
\begin{proposition}\label{S2Prop2}
For any $\lambda \in \mathbb{Y}$ we have
\begin{equation}\label{S2eq1} 
J_{\lambda}(\mathfrak{a}^N)={\bf 1}_{\{ \ell(\lambda) \leq N\}}\cdot \mathfrak{a}^{|\lambda|}\prod_{\square\in{\lambda}}\frac{N\theta+a'(\square)-{\theta}l'(\square)}{a(\square)+{\theta}l(\square)+\theta} \mbox{ and }
J_{\lambda}(\mathfrak{r}_{s})=s^{|\lambda|}\theta^{|\lambda|}\prod_{\square\in\lambda}\frac{1}{a(\square)+\theta{l(\square)+\theta}}.
\end{equation}
\end{proposition}

\begin{proposition}\label{S2Prop3}
For any $\lambda, \mu \in \mathbb{Y}$ we have
\begin{equation}\label{S2eqJ}
\begin{split}
J_{\lambda/\mu} (\mathfrak{a}^1)= &{\bf 1}_{\{ \mu \prec \lambda\}} \cdot \mathfrak{a}^{|\lambda| - |\mu|} \prod_{1 \leq i \leq j \leq k-1} \frac{(\mu_i - \mu_j + \theta(j-i) + \theta)_{\mu_j - \lambda_{j+1}}}{(\mu_i - \mu_j + \theta(j-i) + 1)_{\mu_j - \lambda_{j+1}}} \cdot \\
&\frac{(\lambda_i - \mu_j + \theta(j-i) + 1)_{\mu_j - \lambda_{j+1}}}{(\lambda_i - \mu_j + \theta(j-i) + \theta)_{\mu_j - \lambda_{j+1}}},
\end{split}
\end{equation}
where $k$ is any integer satisfying $\ell(\lambda) \leq k$. When $\mu$ differs from $\lambda$ by one box $\lambda = \mu \sqcup (i,j)$ the above formula can be simplified to read in terms of $\widetilde{J}_{\lambda / \mu}$
\begin{equation}\label{JackSingle}
\widetilde{J}_{\lambda/\mu}(\mathfrak{a}^1) = \mathfrak{a}\theta \cdot \prod_{r = 1}^{i-1} \frac{a(r,j;\mu) + \theta(i - r + 1)}{a(r,j;\mu) + \theta(i - r)} \frac{a(r,j;\mu) + 1 + \theta(i - r - 1)}{a(r,j;\mu) + 1+  \theta(i - r)}.
\end{equation}
\end{proposition}

We also recall the following summation formula for Jack polynomials (this is Proposition 2.5 in \cite{GS1}).
\begin{proposition}\label{S2Prop4}
Let $\rho_1,\rho_2$ be two specializations such that the series $\sum_{k = 1}^\infty \frac{p_k(\rho_1)p_k(\rho_2)}{k}$ is absolutely convergent and define
$$H_\theta(\rho_1;\rho_2) = \exp \left( \theta \sum_{k = 1}^\infty \frac{p_k(\rho_1)p_k(\rho_2)}{k}\right).$$
Then we have
\begin{equation}\label{Cauchy}
\sum_{\lambda \in \mathbb{Y}} J_{\lambda}(\rho_1) \widetilde{J}_\lambda(\rho_2) = H_\theta(\rho_1; \rho_2),
\end{equation}
and more generally for any $\nu, \kappa \in \mathbb{Y}$ 
\begin{equation}\label{SkewCauchy}
\sum_{\lambda \in \mathbb{Y}} J_{\lambda/\nu}(\rho_1) \widetilde{J}_{\lambda/\kappa}(\rho_2) = H_\theta(\rho_1; \rho_2)\sum_{\mu \in \mathbb{Y}} J_{\kappa/\mu}(\rho_1) \widetilde{J}_{\nu/\mu}(\rho_2).
\end{equation}
Note that the sum on the right is in fact finite and so well-defined. Part of the statement is that the left side is actually absolutely convergent and numerically equals the right side.
\end{proposition}

\subsection{Jack measures and dynamics}\label{Section2.2}
We start with the definition of Jack probability measures, based on (\ref{Cauchy}).
\begin{definition}\label{defJM}
Let $\rho_1$ and $\rho_2$ be two Jack-positive specializations such that the series $\sum_{k = 1}^\infty \frac{p_k(\rho_1)p_k(\rho_2)}{k}$ is absolutely convergent. 
The Jack probability measure $\mathcal{J}_{\rho_1,\rho_2}(\lambda)$ on $\mathbb{Y}$ is defined through
\begin{equation}
\mathcal{J}_{\rho_1,\rho_2}(\lambda) =\frac{J_{\lambda}(\rho_1)\widetilde{J}_{\lambda}(\rho_2)}{H_{\theta}(\rho_1;\rho_2)},
\end{equation}
where the normalization constant is given by
$$H_{\theta}(\rho_1;\rho_2)=\exp \left(\sum_{k=1}^{\infty}\frac{\theta}{k}p_k(\rho_1)p_k(\rho_2)\right).$$
\end{definition}

\begin{remark}
The construction of probability measures via specializations of symmetric polynomials was first proposed by Okounkov in the context of Schur measures \cite{Ok01}. Since that seminal work, the framework has been extended to more general polynomials leading, in particular, to the Macdonald measures of \cite{BorCor}.
\end{remark}

As an immediate corollary of Propositions \ref{S2Prop2} we obtain the following result (this is Proposition 2.8 in \cite{GS1}).
\begin{proposition}\label{propJM} Let $1^N$ and $\mathfrak{r}_s$ be as in Proposition \ref{S2Prop2}. Then $\mathcal{J}_{1^N;\tau_{s}}(\lambda)=0$ unless $\lambda \in \mathbb{Y}^N$, in which case we have
$$\mathcal{J}_{1^N;\mathfrak{r}_{s}}(\lambda)=e^{-\theta{sN}}s^{|\lambda|}\theta^{|\lambda|}\prod_{\square\in\lambda}\frac{N\theta+a'(\square)-\theta{l'(\square)}}{(a(\square)+\theta{l(\square)}+\theta)(a(\square)+\theta l(\square)+1)}.$$
\end{proposition}

In what follows we will construct a stochastic dynamics on $\mathbb{Y}^N$. Our discussion will follow to large extent Section 2.3 in \cite{GS1}; however, we remark that similar constructions have been made for Schur, $q$-Whittaker and Macdonald polynomials in \cite{Bor11, BorCor, BorFer, BorGorIP}.
\begin{definition}\label{S2union}
Given two specializations $\rho$ and $\rho'$, define their union $(\rho, \rho')$ through the formulas
$$p_k(\rho, \rho') = p_k(\rho) + p_k(\rho'), \mbox{ for $k\geq1$},$$
where $p_k, k\geq 1$ are the Newton power sums as before.
\end{definition}
Let $\rho$ and $\rho'$ be two Jack-positive specializations such that $H_\theta(\rho;\rho') < \infty$. Define the matrices $p_{\lambda \rightarrow \mu}^{\uparrow}$ and $p_{\lambda \rightarrow \mu}^{\downarrow}$ with rows and columns indexed by Young diagrams as follows:
\begin{equation}
\begin{split}
&p_{\lambda \rightarrow \mu}^{\uparrow}(\rho; \rho') = \frac{1}{H_\theta(\rho; \rho')} \frac{J_{\mu}(\rho)}{J_\lambda(\rho)} \widetilde{J}_{\mu / \lambda}(\rho'), \hspace{2mm} \lambda,\mu \in \mathbb{Y}, \hspace{2mm} J_\lambda(\rho) > 0,\\
&p_{\lambda \rightarrow \mu}^{\downarrow}(\rho; \rho') =\frac{J_\mu(\rho)}{J_\lambda(\rho,\rho')} J_{\lambda/ \mu}(\rho'), \hspace{2mm} \lambda,\mu \in \mathbb{Y}, \hspace{2mm} J_\lambda(\rho, \rho') > 0.
\end{split}
\end{equation}
The following results follow from (\ref{skewJ2}) and (\ref{SkewCauchy}) (see also \cite{Bor11, BorCor, BorGorIP} for similar results in the case of Schur, $q$-Whittaker and Macdonald polynomials).
\begin{proposition}\label{transProp}
The matrices $p_{\lambda \rightarrow \mu}^{\uparrow}$ and $p_{\lambda \rightarrow \mu}^{\downarrow}$ are stochastic, i.e. the elements are non-negative and for $\lambda \in \mathbb{Y}$ we have
$$\sum_{\mu \in \mathbb{Y}}p_{\lambda \rightarrow \mu}^{\uparrow}(\rho, \rho') = 1 = \sum_{\mu \in \mathbb{Y}}p_{\lambda \rightarrow \mu}^{\downarrow}(\rho; \rho').$$
If $\mu \in \mathbb{Y}$ and $\rho_1,\rho_2, \rho_3$ are Jack-positive specializations we have
$$\sum_{\lambda \in \mathbb{Y}: \mathcal{J}_{\rho_1; \rho_2}(\lambda) > 0} \mathcal{J}_{\rho_1; \rho_2}(\lambda) p_{\lambda \rightarrow \mu}^{\uparrow}(\rho_2, \rho_3) = \mathcal{J}_{\rho_1,\rho_3; \rho_2}(\mu)$$
$$\sum_{\lambda \in \mathbb{Y}: \mathcal{J}_{\rho_1; \rho_2,\rho_3}(\lambda) > 0} \mathcal{J}_{\rho_1; \rho_2, \rho_3}(\lambda) p_{\lambda \rightarrow \mu}^{\downarrow}(\rho_1; \rho_3) = \mathcal{J}_{\rho_1; \rho_2}(\mu).$$
The matrices $p_{\lambda \rightarrow \mu}^{\uparrow}$ and $p_{\lambda \rightarrow \mu}^{\downarrow}$ satisfy the following commutation relation
$$p_{\lambda \rightarrow \mu}^{\uparrow}(\rho_1, \rho_2; \rho_3)  p_{\lambda \rightarrow \mu}^{\downarrow}(\rho_1; \rho_2)  = p_{\lambda \rightarrow \mu}^{\downarrow}(\rho_1; \rho_2) p_{\lambda \rightarrow \mu}^{\uparrow}(\rho_1, \rho_3) .$$
\end{proposition}

We note that $p_{\lambda \rightarrow \mu}^{\uparrow}(1^N; \mathfrak{r}_t)$ for $t\geq 0$ defines a {\em transition function} on $\mathbb{Y}^N$ (see Chapter 2 in \cite{Lig}). The stochasticity is a consequence of Proposition \ref{transProp}, while the Chapman-Kolmogorov equations follow from (\ref{skewJ2}).  The condition $\lim_{t \rightarrow 0^+}p_{\lambda \rightarrow \lambda}^{\uparrow}(1^N; \mathfrak{r}_t) = p_{\lambda \rightarrow \lambda}^{\uparrow}(1^N; \mathfrak{r}_0)  = 1$ is a consequence of the fact that $\widetilde{J}_{\mu / \lambda}(\mathfrak{r}_t) = 0$ unless $\lambda \subset \mu$ in which case we have
$$\widetilde{J}_{\mu / \lambda}(\mathfrak{r}_t) = \begin{cases}1 &\mbox{ if $\lambda = \mu$, }\\  O(t^{|\mu| - |\lambda|})  &\mbox{ otherwise.}\end{cases} .$$
In particular, we conclude that there is a (unique) Markov chain with the above transition function for a given initial condition. We call the latter process, started from $\varnothing$ (the element $\lambda \in \mathbb{Y}^N$ such that $\lambda_1 = \cdots = \lambda_N = 0$), $X_{disc}^N(s),s\geq 0$ after \cite{GS1} and record some of its properties in a sequence of propositions, whose proof can be found in Section 2.3 of the same paper.
\begin{proposition}\label{JumpRates} The jump rates of the Markov chain $X^N_{disc}$ on $\mathbb{Y}^N$ are given by
\begin{equation}\label{eqJR}
q_{\lambda \rightarrow \mu} = \begin{cases} \cfrac{J_\mu(1^N)}{J_\lambda(1^N)} \hspace{1mm} \widetilde{J}_{\mu / \lambda}(\mathfrak{r}_1), \hspace{2mm} &\mu = \lambda \sqcup \square, \\
- \sum\limits_{ \nu = \lambda \sqcup \square} q_{\lambda \rightarrow \nu}, &\mu = \lambda, \\ 0 &\mbox{otherwise.}  \end{cases}
\end{equation}
Explicitly, for $\mu = \lambda \sqcup (i,j)$ we have that $q_{\lambda \rightarrow \mu}$ equals
\begin{equation}
\begin{split}
 &\cfrac{ \prod\limits_{\square \in \mu} \cfrac{N\theta + a'(\square) - \theta l'(\square)}{a(\square;\mu) + \theta l(\square; \mu) + \theta}}{\prod\limits_{\square \in \lambda} \cfrac{N\theta + a'(\square) - \theta l'(\square)}{a(\square;\lambda) + \theta l(\square;\lambda) + \theta}} \hspace{1mm} \theta \cdot \prod_{k = 1}^{i-1} \frac{a(k,j;\lambda) + \theta(i-k+1)}{a(k,j;\lambda) + \theta(i-k)} \frac{a(k,j;\lambda) + 1 +  \theta(i-k-1)}{a(k,j;\lambda) + 1 +  \theta(i-l)}.
\end{split}
\end{equation}
\end{proposition}

\begin{proposition}\label{poisson}
The process $\left| X_{disc}^N\right|: = \sum_{i = 1}^N \left( X_{disc}^N\right)_i$ is a Poisson process with intensity $N\theta$.
\end{proposition}

\begin{proposition}For any fixed $s > 0$, the law of $X^N_{disc}(s)$ is given by $\mathcal{J}_{1^N; \mathfrak{r}_s}$, which was computed in Proposition \ref{propJM}.
\end{proposition}

\subsection{Multilevel Jack measures and dynamics}

For $n,N \in \mathbb{N}$, $n \leq N$ denote by $\mathbb{GT}^N_n$ the collection of sequences $(\lambda^n \prec \cdots \prec \lambda^N)$ such that $\lambda^i \in \mathbb{Y}^i$ for $i = n,n+1,...,N$. A natural way to generalize $\mathcal{J}_{1^N;\tau_{s}}$ to a measure on $\mathbb{GT}^N_n$ is given in the following definition.
\begin{definition} For $n,N \in \mathbb{N}$, $n \leq N$ and $s \geq 0$ we define the multilevel Jack probability measure $\mathcal{J}^{multi}_{s;n,N}$ on $\mathbb{GT}^N_n$ by
\begin{equation}\label{S2mult1}
\mathcal{J}^{multi}_{s;n,N}(\lambda^n,...,\lambda^N)=\frac{\widetilde{J}_{\lambda^N}(\mathfrak{r}_s)J_{\lambda^N/\lambda^{N-1}}(1^1)\cdots J_{\lambda^{n+1}/\lambda^n}(1^1)J_{\lambda^n}(1^n)}{H_\theta(\mathfrak{r}_s; 1^N)}.
\end{equation}
For future use we also record the following formulas, which can be obtained from equations (\ref{S2eqDual}), (\ref{S2eq1}), (\ref{S2eqJ}) and the definition of $\mathfrak{r}_s$
\begin{equation}\label{S2mult2}
\begin{split}
&\mathcal{J}^{multi}_{s;n,N}(\lambda^n,...,\lambda^N)={\bf 1}_{\{ \ell(\lambda^n) \leq n\}} \cdot e^{-\theta s N}s^{|\lambda^N|}\theta^{|\lambda^N|}  \prod_{\square\in{\lambda^N}}\frac{1}{a(\square)+\theta{l(\square)+1}} \cdot\\
&  \prod_{\square\in{\lambda^n}}\frac{n\theta+a'(\square)-{\theta}l'(\square)}{a(\square)+{\theta}l(\square)+\theta} \cdot \prod_{k = n+1}^N J_{\lambda^k/\lambda^{k-1}}(1^1), \mbox{ with}\\
&J_{\lambda^k/\lambda^{k-1}}(1^1)= {\bf 1}_{\{\lambda^{k-1} \prec \lambda^k\}}\prod_{1\leq{i}\leq{j}\leq{k-1}}\frac{f(\lambda_i^{k-1}-\lambda_j^{k-1}+\theta(j-i))f(\lambda_i^k-\lambda_{j+1}^k+\theta(j-i))}{f(\lambda_i^{k-1}-\lambda_{j+1}^k+\theta(j-i))f(\lambda_i^k-\lambda_j^{k-1}+\theta(j-i))}
\end{split}
\end{equation}
$ \mbox{ where $\lambda^0 = \varnothing$ and $f(z)=\frac{\Gamma(z+1)}{\Gamma(z+\theta)}$}$. Here $\Gamma$ denotes the usual gamma function.
\end{definition}

We summarize some of the properties of $\mathcal{J}^{multi}_{s;n,N}$ in the following lemma.
\begin{lemma}\label{S2LMJ}
Let $\mathcal{J}^{multi}_{s;n,N}$ be as in (\ref{S2mult1}). Then $\mathcal{J}^{multi}_{s;n,N}$ defines a probability measure on $\mathbb{GT}^N_n$, which satisfies the following {\em Jack-Gibbs property}:
\begin{equation}\label{JackGibbs}
\mathcal{J}^{multi}_{s;n,N}(\lambda^n,...,\lambda^{N-1}|\lambda^{N} = \mu) \propto J_{\mu/\lambda^{N-1}}(1^1)\cdots J_{\lambda^{n+1}/\lambda^n}(1^1)J_{\lambda^n}(1^n).
\end{equation}
Furthermore, the projection of $\mathcal{J}^{multi}_{s;n,N}$ to $(\lambda^{n'},...,\lambda^{N'})$ where $n \leq n' \leq N' \leq N$ is given by $\mathcal{J}^{multi}_{s;n',N'}$ and the projection to $\lambda^{N'}$ is $\mathcal{J}_{1^{N'};\mathfrak{r}_{s}}$.
\end{lemma}
\begin{proof}
Let us sum (\ref{S2mult1}) over $\lambda^{n},...,\lambda^{n'-1}$ and over $\lambda^{N' + 1},...,\lambda^{N-1}$. Using (\ref{skewJ2}) we see that the result is 
\begin{equation}\label{S2Leq}
\frac{\widetilde{J}_{\lambda^N}(\mathfrak{r}_s)J_{\lambda^{N}/\lambda^{{N'}}}(1^{N-N'})J_{\lambda^{n' }}(1^{n'})\prod_{k = n'}^{N'-1} J_{\lambda^{k+1}/\lambda^{k}}(1^1)}{H_\theta(\mathfrak{r}_s; 1^N)}.
\end{equation}
If we set $n'=N' = N$ we see that the above is equal to $\frac{\widetilde{J}_{\lambda^N}(\mathfrak{r}_s)J_{\lambda^{N }}(1^{N})}{H_\theta(\mathfrak{r}_s; 1^N)}$, which by (\ref{Cauchy}) proves that $\mathcal{J}^{multi}_{s;n,N}$ is indeed a probability measure. The fact that $\mathcal{J}^{multi}_{s;n,N}$ satisfies the the Jack-Gibbs property is immediate from its definition.

If we sum (\ref{S2Leq}) over $N$ and use (\ref{SkewCauchy}) we will obtain
$$H_\theta( 1^{N - N'};\mathfrak{r}_s)\frac{J_{\lambda^{{N'}}}(\mathfrak{r}_s)J_{\lambda^{n' }}(1^{n'})\prod_{k = n'}^{N'-1} J_{\lambda^{k+1}/\lambda^{k}}(1^1)}{H_\theta(\mathfrak{r}_s; 1^N)}.$$
Using that $H_\theta( 1^{k}; \mathfrak{r}_s) = H_\theta(  \mathfrak{r}_s; 1^{k}) = e^{\theta sk}$, we recognize the above as  $\mathcal{J}^{multi}_{s;n',N'}(\lambda^{n'},...,\lambda^{N'})$, which proves the second part of the lemma. The last part follows from the second one when $n' = N'$.
\end{proof}
\begin{remark}
The measure $\mathcal{J}^{multi}_{s;1,N}$ is the ascending Jack process with specialization $\rho = \mathfrak{r}_s$ (see Definition 2.12 in \cite{GS1}).
\end{remark}
\begin{remark}
When $\theta = 1$, (\ref{JackGibbs}) means that the conditional distribution of $\lambda^n,...,\lambda^{N-1}$ is {\em uniform} on the polytope defined by the interlacing conditions.
\end{remark}

Our next goal is to define a stochastic dynamics on $\mathbb{GT}^N_n$, whose restriction to level $k \in \{n,...,N\} $ has the same law as $X^k_{disc}(s)$, $s\geq 0$. Analogues of this construction were done in \cite{Bor11, BorCor, BorFer, BorGorIP} and they are based on an idea going back to \cite{DF}, which allows one to couple dynamics of Young diagrams of different sizes. 

Let us fix $\epsilon > 0$. For $k = n,...,N$ we introduce the following notation
\begin{equation*}
\begin{split}
P_k(\lambda, \mu)  &=  p_{\lambda \rightarrow \mu}^{\uparrow}(1^k; \mathfrak{r}_\epsilon)\mbox{ for $\lambda, \mu \in \mathbb{Y}^k$,} \\
 \Lambda^k_{k-1}(\lambda, \mu) &=  p_{\lambda \rightarrow \mu}^{\downarrow}(1^{k-1}; 1)  \mbox{ for $\lambda \in \mathbb{Y}^k$ and $\mu \in\mathbb{Y}^{k-1}$.} 
\end{split}
\end{equation*}
It follows from Proposition \ref{transProp} that $P_k\Lambda^k_{k-1} = \Lambda^k_{k-1}P_{k-1}$ for $k = n+1,...,N$ and we denote this matrix by $\Delta^k_{k-1}$. For $X = (x_n,...,x_N), Y = (y_n,...,y_N)\in \mathbb{GT}^N_n$ we define
\begin{equation}\label{MTP}
\mathbb{P}^{(N,n)}(X,Y) = P_n(x_n,y_n) \prod_{k = n+1}^N \frac{P_k(x_k,y_k) \Lambda^k_{k-1}(y_k, y_{k-1})}{\Delta^k_{k-1}(x_k,y_{k-1})}
\end{equation}
if $\prod_{k = n+1}^N\Delta^k_{k-1}(x_k,y_{k-1}) > 0$ and $0$ otherwise. As shown in Section 2.2 in \cite{BorFer}, we have that $\mathbb{P}^{(N,n)}$ is a stochastic matrix and so we can define with it a {\em discrete time} Markov chain on $\mathbb{GT}^N_n$. 
\begin{remark}\label{SU} One way to think of $\mathbb{P}^{(N,n)}$ is as follows: Starting from $X = (x_n,...,x_N)$ we first choose $y_n$ according to the transition matrix $P_n(x_n,y_n)$, then choose $y_{n+1}$ using $\frac{P_{n+1}(x_{n+1}, y_{n+1})\Lambda^{n+1}_n(y_{n+1}, y_n)}{\Delta^{n+1}_n(x_{n+1}, y_n)}$, which is the conditional distribution of the middle point in successive applications of $P_{n+1}$ and $\Lambda^{n+1}_n$, provided that we start from $x_{n+1}$ and finish at $y_n$, after that we choose $y_{n+2}$ using the conditional distribution of the middle point of successive applications of $P_{n+2}$ and $\Lambda^{n+2}_{n+1}$ provided we start at $x_{n+2}$ and finish at $y_{n+1}$ and so on. One calls this procedure of obtaining $Y$ from $X$ a {\em sequential update}.
\end{remark}

We denote the Markov chain with transition matrix $\mathbb{P}^{(N,n)}$, which is started from $ (\varnothing, ..., \varnothing) \in \mathbb{GT}^N_n$ (i.e. the sequence $(\lambda^n ,...,\lambda^N)$ with $\lambda^j_i = 0$ for $1 \leq i \leq j$ and $n \leq j \leq N$), by $\hat{X}^{multi}_{n,N}(m; \epsilon), m\in \mathbb{N}_0$. The following proposition contains the main property of $\hat{X}^{multi}_{n,N}(m; \epsilon), m\in \mathbb{N}_0$ that we will need.
\begin{proposition}\label{S2PD} Let $\hat{X}^{multi}_{n,N}(m; \epsilon), m\in \mathbb{N}_0$ be as above.
\begin{enumerate}
\item The restriction of $\hat{X}^{multi}_{n,N}(m; \epsilon), m\in \mathbb{N}_0$ to levels $\{n', n'+1,...,N'\}$ with $n \leq n' \leq N' \leq N$ has the same law as $\hat{X}^{multi}_{n',N'}(m; \epsilon), m\in \mathbb{N}_0$. 
\item The law of $\hat{X}^{multi}_{n,N}(m; \epsilon)$ for fixed $m \geq 0$ is given by $\mathcal{J}^{multi}_{m\epsilon;n,N}$.
\item For each $k \in \{n,...,N\}$ we have that $\hat{X}^{multi}_{n,N}(m; \epsilon),  m\in \mathbb{N}_0$, restricted to level $k$, has the same law as $X^k_{disc}(m\epsilon),  m\in \mathbb{N}_0$.
\end{enumerate}
\end{proposition}
\begin{proof}
We know that the restriction of $\hat{X}^{multi}_{n,N}(m; \epsilon), m\in \mathbb{N}_0$ to levels $\{n,...,N'\}$ has the same law as $\hat{X}^{multi}_{n,N'}(m; \epsilon), m\in \mathbb{N}_0$, because of the sequential update property of the transition matrix $\mathbb{P}^{(N,n)}$ (see Remark \ref{SU}) and this holds for any initial conditions. Thus it suffices to show the first property when $N' = N$. The first statement now follows from Proposition 2.2 in \cite{BorFer} applied to the following setting
$$S^* = \mathbb{Y}^{n +1} \times \cdots \times \mathbb{Y}^N, \hspace{5mm} S = \mathbb{Y}^n, \hspace{5mm} \Lambda((x_{n+1},...,x_{N}), x_n) = \Lambda^{n+1}_n(x_{n+1}, x_n)$$
$$P^*(X,Y) =  P_{n+1}(x_{n+1},y_{n+1}) \prod_{k = n+2}^N \frac{P_k(x_k,y_k) \Lambda^k_{k-1}(y_k, y_{k-1})}{\Delta^k_{k-1}(x_k,y_{k-1})}, \hspace{5mm}  P(x,y) = P_1(x,y),$$
where $X = (x_{n+1},...,x_{N})$ and $Y = (y_{n+1},...,y_{N})$ are in $S^*$. The essential ingredients we need are the commutation relation $P_k\Lambda^k_{k-1} = \Lambda^k_{k-1}P_{k-1}$ and the fact that the delta mass at $(\varnothing, ..., \varnothing)$ satisfies the Jack-Gibbs property (see (\ref{JackGibbs})). The second property is a consequence of Proposition 2.5 in \cite{BorCor} and again relies on the commutation relation $P_k\Lambda^k_{k-1} = \Lambda^k_{k-1}P_{k-1}$ and the fact that the delta mass at $(\varnothing, ..., \varnothing)$ satisfies the Jack-Gibbs property. 

The third property follows from the first by setting $n' = k = N'$, since then $\hat{X}^{multi}_{k,k}(m; \epsilon), m\in \mathbb{N}_0$ is defined entirely in terms of $P_k$, which is the transition matrix for $X^N_{disc}(m\epsilon), m \geq 0$ and the distribution of the two chains at $m = 0$ is the same.
\end{proof}

As discussed in Section 2.3 \cite{GS1}, the continuous time processes $\hat{X}^{multi}_{n,N}(\lfloor s \epsilon^{-1} \rfloor; \epsilon),s \geq 0$, weakly converge to a fixed continuous time process. We will only sketch the main ideas and refer the reader to Section 2.7 in \cite{BorFer} and Section 2.3.3 in \cite{BorCor} for a careful treatment of this limit transition in the context of Schur and Macdonald polynomials.

Using that $\widetilde{J}_{\kappa/\lambda}(\mathfrak{r}_\epsilon)$ is of order $\epsilon^{|\kappa|-|\lambda|}$ as $\epsilon \rightarrow 0^+$ one obtains that 
$$\mathbb{P}^{(N,n)}(X,Y) \approx {\bf 1}_{\{X = Y\}} + \epsilon Q(X,Y) + O(\epsilon^2),$$
where $Q$ is a certain matrix with rows and columns parametrized by elements in $\mathbb{GT}^N_n$. Let $X = (\lambda^n,...,\lambda^N)$ and $Y = (\mu^n,...,\mu^N)$, then the entry $Q(X,Y)$ can be found as follows. Suppose that $X \neq Y$ and there exist integers $A,B,C$ and $x$ such that
$$1 \leq A \leq B, \hspace{5mm} n \leq B \leq N, \hspace{5mm} 0 \leq C \leq N-B,$$
\begin{equation*}
\begin{split}
&\lambda^B_A = \lambda_A^{B+1} = \cdots = \lambda_A^{B+C} = x,\\
&\mu^B_A = \mu_A^{B+1} = \cdots = \mu_A^{B+C} = x + 1,
\end{split}
\end{equation*}
$\mu^m_k = \lambda_k^m$ for all other values of $(k,m)$. Then we have that 
\begin{equation}\label{Q1}
Q(X,Y) = \begin{cases} \widetilde{J}_{\mu^B / \lambda^B}(\mathfrak{r}_1) \cfrac{J_{\mu^B / \lambda^{B-1}}(1^1)}{J_{\lambda^B / \lambda^{B-1}}(1^1)} &\mbox{ if } B \geq n+1\\  \widetilde{J}_{\mu^B / \lambda^B}(\mathfrak{r}_1) \cfrac{J_{\mu^B }(1^n)}{J_{\lambda^B }(1^n)} &\mbox{ if }B = n.\end{cases}
\end{equation}
In all other cases when $X\neq Y$ we have $Q(X,Y) = 0$ and the diagonal entries are given by
\begin{equation}\label{Q2}
Q(X,X) = - \sum_{Y \neq X} Q(X,Y).
\end{equation}
The above description of $Q$ implies that in each row at most $N(N+1)/2$ off-diagonal entries are non-zero and we note that they are uniformly bounded (independently of $X$ and $Y$). Indeed, by (\ref{JackSingle}) we have that $J_{\mu / \lambda}(1^1)$ is bounded from above and below for any $\mu, \lambda \in \mathbb{Y}^N$ and $\lambda \prec \mu$. In addition, we have that $\mathcal{J}_{\mu / \lambda}(1^1) = \mathcal{J}_{\mu / \lambda}(\mathfrak{r}_1)$, whenever $\mu = \lambda \sqcup \square$. Thus the top expressions in (\ref{Q1}) are all bounded away from $0$ and $\infty$. From Proposition \ref{poisson} we know that the sum over $\mu^B$ of the bottom expression in (\ref{Q1}) equals $\theta n$. Overall, this implies that $Q(X,X)$ are uniformly bounded from below and so by Theorem 2.37 in \cite{Lig} there exists a (unique) continuous time Markov chain, whose infinitesimal generator (or $Q$-matrix) is given by $Q$.  

From Lemma 2.21 in \cite{BorFer} we have that $\lim_{\epsilon \rightarrow 0^{+}}\left( \mathbb{P}^{(N,n)}\right)^{\lfloor s\epsilon^{-1} \rfloor} = \exp(t Q)$, which in turn implies the weak convergence of  $\hat{X}^{multi}_{n,N}(\lfloor s \epsilon^{-1} \rfloor; \epsilon),s \geq 0$ to the continuous time Markov chain with infinitesimal generator $Q$, started from $(\varnothing,...,\varnothing)$. We call the latter process $X^{multi}_{n,N}(s); s\geq 0$ and isolate its definition below for future reference.
\begin{definition}\label{mainDef2} For $N,n \in \mathbb{N}$ with $N \geq n$ we define the multilevel Jack process to be the continuous time Markov chain $X^{multi}_{n,N}(s); s\geq 0$ on $\mathbb{GT}^N_n$ with initial condition $(\varnothing, ... \varnothing)$ and infinitesimal generator ($Q$-matrix), given by (\ref{Q1}) and (\ref{Q2}).
\end{definition} 
\begin{remark}
For $n = 1$ the process $X^{multi}_{n,N}(s), s\geq 0$ agrees with the process $X^{multi}_{disc}(s), s\geq 0$ of Definition 2.27 in \cite{GS1}.
\end{remark}

An alternative (less formal) description of $X^{multi}_{n,N}(s), s\geq 0$ can be given as follows. Each of the coordinates (particles) $\lambda^B_k$ such that $\lambda^B_k < \min (\lambda^B_{k-1}, \lambda^{B-1}_{k-1})$ has its own exponential clock with rate
\begin{equation*}
q( \lambda^B_k) = \begin{cases} \widetilde{J}_{\lambda^B \sqcup \square_k / \lambda^B}(\mathfrak{r}_1) \cfrac{J_{\lambda^B \sqcup \square_k / \lambda^{B-1}}(1^1)}{J_{\lambda^B / \lambda^{B-1}}(1^1)} &\mbox{ if $B \geq n + 1$},\\ \widetilde{J}_{\lambda^B \sqcup \square_k / \lambda^B}(\mathfrak{r}_1) \cfrac{J_{\lambda^B \sqcup \square_k }(1^n)}{J_{\lambda^B }(1^n)} &\mbox{ if $B = n$}.\end{cases} 
\end{equation*}
All clocks are independent and $\square_k$ denotes the box at location $\lambda^B_k+1$. When the $\lambda^B_k$ clock rings we find the longest string $\lambda_k^B = \lambda_k^{B+1} = \cdots = \lambda_k^{B+C}$ and move all coordinates in the string to the right by one. I.e. when the clock rings the particle $\lambda_k^B$ jumps to the right by $1$ and if the jump violates interlacing with the top rows, it pushes all particles above it to the right to restore interlacing.

We isolate some properties of $X^{multi}_{n,N}(s), s\geq 0$ in a proposition below. The following are obtained from Proposition \ref{S2PD} by a limit transition.
 \begin{proposition}\label{S2PC}
Let $X^{multi}_{n,N}(s), s\geq 0$ be as in Definition \ref{mainDef2}. 
\begin{enumerate}
\item The restriction of $X^{multi}_{n,N}(s), s\geq 0$ to levels $\{n', n'+1,...,N'\}$ with $n \leq n' \leq N' \leq N$ has the same law as ${X}^{multi}_{n',N'}(s), s\geq 0$.
\item The law of $X^{multi}_{n,N}(s)$ at a fixed time $s \geq 0$ is given by $\mathcal{J}^{multi}_{s;n,N}$.
\item For each $k \in \{n,...,N\}$ we have that ${X}^{multi}_{n,N}(s), s\geq 0$, restricted to level $k$, has the same law as $X^k_{disc}(s), s \geq 0$.
\end{enumerate}
\end{proposition}

\section{Fixed time limit}\label{Section3}
Using the setup of Section \ref{Section2} we can now make a formal definition of the process $X(s;N)$ that we discussed in Section \ref{Section1.2}.
\begin{definition}\label{mainDef} For $\theta > 0$ and $N \in \mathbb{N}$, we define the process $X(s;N) = (X_i^j(s;N): 1 \leq i \leq j \leq N), s\geq 0$ to be the continuous time Markov chain on the space of interlacing arrays $\mathbb{GT}_1^N $, whose distribution is given by the multilevel Jack process $X^{multi}_{1,N}(s); s\geq 0$ of Definition \ref{mainDef2}.
\end{definition}

In this section we prove Theorem \ref{Theorem1}. The main argument is presented in Sections \ref{Section3.2} and \ref{Section3.3}. In the section below we prove several results that will be used in the proof. In the remainder of this paper we write $\xrightarrow{L^1}$, $\xrightarrow{D}$ and $\xrightarrow{\mathbb{P}}$ for convergence in $L^1$, in distribution and in probability respectively.

\subsection{Preliminaries for Theorem \ref{Theorem1}}\label{Section3.1}
In this section we prove several asymptotic results about the distribution of the top row of $X(s;N)$ when $s$ and $N$ become large. These are given in Lemma \ref{LSep} and will be used later in the proof of Theorem \ref{Theorem1}. In order to show Lemma \ref{LSep} we will require several additional results, which we present below. The proof of the following statements relies on an identification of $\mathcal{J}_{1^N;\mathfrak{r}_{s}}$ with the discrete $\beta$-ensemble of \cite{BGG} and is presented in Section \ref{Section5}. 

\begin{proposition}\label{empProp}
Let $X(s;N)$ be as in Definition \ref{mainDef}. For $s \geq 0$ and $t > 0$ define the measures
$$\mu^{s,t}_N= \frac{1}{N}\sum_{i = 1}^N \delta \left( \frac{X^N_i(tN + s ;N) + (N-i+1)}{N}\right).$$ Then there exists a deterministic measure $\mu^t$, such that $\mu^{s,t}_N \Rightarrow \mu^t$ as $N \rightarrow \infty$, in the sense that for any bounded continuous function $f$, we have the following convergence in probability:
$$\lim_{N \rightarrow \infty} \int_{\mathbb{R}}{f(x)d\mu^{s,t}_N(x)}={\int_{\mathbb{R}}{f(x)d\mu^t}(x)}.$$
\end{proposition}
The measure $\mu^t$ will be explicitly computed in Section \ref{Section5.2}; however, we summarize the properties we will need in this section.
\begin{enumerate}
\item The measure $\mu^t$ is compactly supported on the interval $[0, b_t]$, where $b_t = \theta (1 + \sqrt{t})^2$.
\item The Stieltjes transform\footnote{The Stieltjes transform of a measure $\mu$ on $\mathbb{R}$ is defined by $G_\mu(z) = \int_\mathbb{R} \frac{1}{z - x}d\mu(x)$ for $z \not \in $ supp($\mu$).} of the measure $\mu^t$, satisfies $\lim_{\epsilon \rightarrow 0} G_{\mu^t}(b_t + \epsilon) = \theta^{-1} \log(1 + t^{-1/2})$.
\end{enumerate}

\begin{proposition}\label{conProp}
Let $X(s;N)$ be as in Definition \ref{mainDef} and fix $s \geq 0$ and $t > 0$. Define $\ell_i = X^N_{N-i+1}(tN + s;N) + \theta \cdot i$ for $i = 1,...,N$. Then we have the following 
\begin{equation}\label{ObsF}
\lim_{N \rightarrow \infty} \mathbb{E} \left[\prod_{j = 1}^{N-1} \left( 1 - \frac{1}{\ell_N - \ell_j} \right) \left( 1 - \frac{2\theta - 1}{\ell_N - \ell_j + \theta - 1}\right) \right] =  \frac{t}{(\sqrt{t} + 1)^2}, 
\end{equation}
\begin{equation}\label{S3Right}
\frac{\ell_N}{N}  \xrightarrow{L^1} \theta(1 + \sqrt{t})^2\mbox{ as $N \rightarrow \infty$}.
\end{equation}
\end{proposition}

Proposition \ref{empProp} and the two properties that follow it are proved at the end of Section \ref{Section5.2} and Proposition \ref{conProp} is proved in Section \ref{Section5.3}. We now turn to the main result of this section.
\begin{lemma}\label{LSep} Let $X(s;N)$  be as in Definition \ref{mainDef} with $\theta \geq 1$ and fix $s \geq 0$ and $t > 0$. Define $\ell_i = X^N_{N-i+1}(tN + s ;N) + \theta \cdot i$ for $i = 1,...,N$ and $m_j = X^{N-1}_{N-j}(tN + s ;N-1) + \theta \cdot j$ for $j = 1,...,N-1$. Then we have the following convergence results as $N \rightarrow \infty$:
\begin{equation}
\mbox{1.}\hspace{1mm}\cfrac{1}{\ell_N-\ell_{N-1}} \xrightarrow{\mathbb{P}} {0},  \hspace{3mm} \mbox{2.}\sum\limits_{j=1}^{N-1}\cfrac{1}{\ell_N-\ell_j} \xrightarrow{\mathbb{P}} \frac{\log(1+t^{-1/2})}{\theta},  \hspace{3mm}  \mbox{3.}\sum\limits_{j=1}^{N-2}\cfrac{1}{\ell_N-m_j} \xrightarrow{\mathbb{P}} \frac{\log(1+t^{-1/2})}{\theta}.
\end{equation}
\end{lemma}
\begin{proof}
Notice that $\frac{1}{\ell_N - \ell_j} $ and $\frac{2\theta - 1}{\ell_N - \ell_j + \theta - 1}$ are both between $0$ and $1$ (since $\ell_N - \ell_i \geq \ell_N - \ell_{N-1} \geq \theta \geq 1$).

We next observe that if $x \in [0,1]$ then we have $0 \leq (1 - x) \leq \exp( - x - x^2/2)$. Indeed, setting $f(x) = 1 - x - \exp( - x - x^2/2)$ we see that $f'(x) = (1+x)\exp(-x-x^2/2) - 1$ and we have $f'(x) \leq 0 \iff 1 +x \leq \exp(x + x^2/2)$, which is true for all $x \in \mathbb{R}$. So $f$ is a decreasing function on $\mathbb{R}$ and as $f(0) = 0$ we conclude that $f(x) \leq 0$ for $x \in [0,1]$. This proves that $1 - x \leq \exp( - x - x^2/2)$ for $x \in [0,1]$, while $0 \leq 1 - x$ is obvious. The latter implies that
\begin{equation}\label{S3eq1}
 \prod_{j=1}^{N-1}\left(1-\frac{1}{\ell_N-\ell_j}\right)\left(1-\frac{2\theta-1}{\ell_N-\ell_j+\theta-1}\right) \leq X_NY_N,
\end{equation}
where 
\begin{equation}\label{XnYn}
\begin{split}
&X_N= \exp\left[-\sum_{j=1}^{N-1}\left(\frac{1}{\ell_N-\ell_j}+\frac{2\theta-1}{\ell_N-\ell_j+\theta-1}\right)\right] \mbox{ and }\\
&Y_N=\exp\left[-\frac{1}{2}\sum_{j=1}^{N-1}\left(\frac{1}{(\ell_N-l_j)^2}+\frac{(2\theta-1)^2}{(\ell_N-\ell_j-1+\theta)^2}\right)\right].
\end{split}
\end{equation}

In what follows we prove that $X_N \xrightarrow{\mathbb{P}} t(\sqrt{t} + 1)^{-2}$ and $Y_N \xrightarrow{\mathbb{P}} 1 $ as $N \rightarrow \infty$. We observe that $X_N, Y_N \in [0,1]$ and so it suffices to show that any weak subsequential limit of $(X_N,Y_N)$ equals $( t(\sqrt{t} + 1)^{-2}, 1)$. In particular, by possibly passing to a subsequence, we may assume that $X_N$ and $Y_N$ are defined on the same probability space, are converging to random variables $X$ and $Y$ a.s. and we want to show that $X = t(\sqrt{t} + 1)^{-2}$ and $Y = 1$ a.s.

From the Bounded Convergence theorem we know that $\mathbb{E}[XY] = \lim_{N \rightarrow \infty} \mathbb{E}[X_NY_N]$, which together with (\ref{S3eq1}) and (\ref{ObsF}) implies that 
\begin{equation}\label{S3P1}
\mathbb{E}[XY] \geq t(\sqrt{t} + 1)^{-2}.
\end{equation}

Take $\epsilon>0,\delta>0$. By (\ref{S3Right}) we have $\mathbb{P}\left(\frac{\ell_N}{N}\geq{\theta(\sqrt{t}+1)^2}+\frac{\delta}{2}\right)<\epsilon$ for $N$ large enough. On the event $\left\{\frac{\ell_N}{N}\leq{\theta(\sqrt{t}+1)^2}+\frac{\delta}{2}\right\}$ we have for $N$ large enough:
$$\sum_{j=1}^{N-1}\frac{1}{\ell_N-\ell_j}+\frac{2\theta-1}{\ell_N-\ell_j+\theta-1}\geq \frac{1}{N}\sum_{j=1}^{N-1}\frac{1}{\frac{\delta}{2}+\theta(\sqrt{t}+1)^2-\frac{\ell_j}{N}}+\frac{2\theta-1}{\frac{\delta}{2}+\theta(\sqrt{t}+1)^2+\frac{\theta-1}{N}-\frac{\ell_j}{N}}$$
$$\geq  \frac{2\theta}{N}\sum_{j=1}^{N-1}\frac{1}{\delta+\theta(\sqrt{t}+1)^2-\frac{\ell_j}{N}}=\frac{2\theta}{N}\sum_{j=1}^{N}\frac{1}{\delta+\theta(\sqrt{t}+1)^2-\frac{\ell_j}{N}}-\frac{2\theta}{N}\frac{1}{\delta+\theta(\sqrt{t}+1)^2-\frac{\ell_N}{N}}.$$
The first term in the last expression converges to $2\theta \cdot G_{\mu^t}(b_t+\delta)$ by Proposition \ref{empProp} and the discussion below it. On the event $\left\{\frac{\ell_N}{N}\leq {\theta (\sqrt{t}+1)^2}+\frac{\delta}{2}\right\}$ the second term is bounded by
 $$\frac{2\theta}{N}\frac{1}{\delta +\theta (\sqrt{t}+1)^2-\frac{\ell_N}{N}}\leq{\frac{4\theta}{N\delta}}.$$ 
The above suggests that for any $\delta, \epsilon > 0$ we have
$$\liminf_{N \rightarrow \infty} \mathbb{P}(X_N \leq \exp(-2\theta G_{\mu^t}(b_t+\delta)) \geq 1- \epsilon .$$
Consequently, $\mathbb{P}(X \leq  \exp(-2\theta G_{\mu^t}(b_t+\delta)) \geq 1 -\epsilon$. Since $\delta, \epsilon > 0$ are arbitrary, we conclude that 
\begin{equation}\label{S3P2}
X \leq  t(\sqrt{t} + 1)^{-2},
\end{equation}
where we used the second property of $\mu^t$ after Proposition \ref{empProp} and the inequality is a.s. To summarize, we have $0 \leq X \leq t(\sqrt{t} + 1)^{-2}$ and $0 \leq Y \leq 1$ a.s. (recall that $X_N, Y_N \in [0,1]$ a.s.), while the product satisfies $\mathbb{E}[XY] \geq t(\sqrt{t} + 1)^{-2}$. This is possible only if $X = t(\sqrt{t} + 1)^{-2}$ a.s. and $Y = 1$ a.s. We thus conclude that $X_N \xrightarrow{\mathbb{P}} t(\sqrt{t} + 1)^{-2}$ and $Y_N \xrightarrow{\mathbb{P}} 1 $ as $N \rightarrow \infty$. In particular, as $N \rightarrow \infty$ we get
\begin{equation}\label{S3eq2}
\begin{split}
&\sum_{j=1}^{N-1}\frac{1}{(\ell_N-\ell_j)^2}+\frac{(2\theta-1)^2}{(\ell_N-\ell_j+\theta-1)^2}\xrightarrow{\mathbb{P}} 0 \mbox{, and} \\ 
&\sum_{j=1}^{N-1}\frac{1}{\ell_N-\ell_j}+\frac{2\theta-1}{\ell_N-\ell_j+\theta-1}\xrightarrow{\mathbb{P}} 2 \cdot {\log}\left(1+\frac{1}{\sqrt{t}}\right).
\end{split}
\end{equation}

From (\ref{S3eq2}) it is obvious that $\frac{1}{(\ell_N-\ell_j)^2}\xrightarrow{\mathbb{P}}{0}$, and so $\frac{1}{\ell_N-\ell_j}\xrightarrow{\mathbb{P}}{0}$. This proves 1. We next observe that
\begin{equation}\label{S3eq3}
 \sum_{j=1}^{N-1}\frac{1}{\ell_N-\ell_j}+\frac{2\theta-1}{\ell_N-\ell_j+\theta-1}=\sum_{j=1}^{N-1}\frac{2\theta}{\ell_N-\ell_j}+ \sum_{j=1}^{N-1}\frac{(1-\theta)(2\theta-1)}{(\ell_N-\ell_j)(\ell_N-\ell_j+\theta-1)}. 
\end{equation}
Consequently, (\ref{S3eq2}) and (\ref{S3eq3}) show that  $2\theta\sum_{j=1}^{N-1}\frac{1}{\ell_N-\ell_j}\to{2{\log}\left(1+\frac{1}{\sqrt{t}}\right)}$. This proves 2.

Notice that because of the interlacing property of the top two rows of $X(s;N)$, we have that $\ell_1 \leq m_{1} \leq \ell_{2} \leq \cdots \leq \ell_{N-1} \leq m_{N-1} \leq \ell_N$. This implies that
$$\sum_{j=1}^{N-1}\frac{1}{\ell_N-\ell_j}-\frac{1}{\ell_N-\ell_{N-1}}=\sum_{j=1}^{N-2}\frac{1}{\ell_N-\ell_j}\leq\sum_{j=1}^{N-2}\frac{1}{\ell_N-m_j}\leq \sum_{j=1}^{N-2}\frac{1}{\ell_N-\ell_{j+1}}\leq{\sum_{j=1}^{N-1}\frac{1}{\ell_N-\ell_j}}.$$
From the above it is clear that 3. follows from 1.and 2.
\end{proof}

\subsection{Two row analysis}\label{Section3.2}
We start our proof of Theorem \ref{Theorem1} by first showing it holds when $k = 1$. The general statement will be proved in the next section. Our approach here closely follows ideas from \cite{GS2}, where the authors proved an analogous result for $\beta$-Dyson Brownian motion. We summarize the statement we will prove in a lemma.
\begin{lemma}\label{KeyLemma1}
 Let $X(s;N)$ as in Definition \ref{mainDef} with $\theta \geq 1$ and fix $s \geq 0$ and $t > 0$. Define $\ell_i = X^N_{N-i+1}(tN + s ;N) + \theta \cdot i$ for $i = 1,...,N$ and $m_j = X^{N-1}_{N-j}(tN + s ;N-1) + \theta \cdot j$ for $j = 1,...,N-1$. Then $\displaystyle{\ell_N-m_{N-1}}\xrightarrow{D}Z$ as $N \rightarrow \infty$, where 
 $$\mathbb{P}( Z = n+\theta) = (1- p)^{-\theta}\frac{\Gamma(n+\theta)}{\Gamma(n+1)\Gamma(\theta)} p^n, \hspace{1mm} n \in \mathbb{Z}_{\geq 0} \mbox{, and } p = \frac{\sqrt{t}}{1 + \sqrt{t}}.$$. 
\end{lemma}

\begin{proof} Let us fix $\lambda \in \mathbb{Y}^N$ and $\mu \in \mathbb{Y}^{N-1}$ such that $\mu \prec \lambda$. From Proposition \ref{S2PC} we know that 
\begin{equation}\label{S3L1}
\mathbb{P}( X^N(tN + s; N) = \lambda, X^{N-1}(tN + s; N) = \mu) = \mathcal{J}^{multi}_{tN + s;N-1,N}(\mu, \lambda).
\end{equation}
We introduce the following useful notation
$$z_N=\ell_N-m_{N-1} - \theta, \mbox{ and } f_N(k) = \mathbb{P}(z_N=k|\ell_1,...,\ell_N,m_1,...,m_{N-2}) \mbox{, for $k\geq 0$}.$$
Combining (\ref{S3L1}) with (\ref{S2eq1}) and (\ref{S2mult2}) we see that for some $c>0$ (depending on $\ell_1,...,\ell_N$ and $m_1,...,m_{N-2}$) we have
\begin{equation}\label{S3eq4}
f_N(k) = c \cdot \frac{\Gamma(\theta+k)}{\Gamma(1+k)}\prod_{i=1}^{N-2}(\ell_N-m_i-k-\theta)\prod_{i=1}^{N-1}\frac{\Gamma(\ell_N-\ell_i-k)}{\Gamma(\ell_N-\ell_i-k+1-\theta)},
\end{equation}
if $0 \leq k \leq \ell_N - \ell_{N-1} - \theta$ and $0$ otherwise. We want to prove the following statements
\begin{equation}\label{S3Rat}
\frac{f_N(k)}{f_N(0)} \xrightarrow{\mathbb{P}} {p^k} \frac{\Gamma(k+\theta)}{k!\Gamma(\theta)}  \mbox{, for $k \geq 0$ and }  f_N(0)\xrightarrow{\mathbb{P}} (1-p)^{\theta}  \mbox{ as $N \rightarrow \infty$}, 
\end{equation}
where $p = \frac{\sqrt{t}}{1 + \sqrt{t}}$. If (\ref{S3Rat}) is true, then we have that $f_N(k) \xrightarrow{\mathbb{P}} (1- p)^{-\theta}\frac{\Gamma(k+\theta)}{\Gamma(k+1)\Gamma(\theta)}$ as $N \rightarrow \infty$. Taking expectations on both sides (this is allowed by the Bounded Convergence Theorem), we conclude that
$$\lim_{N \rightarrow \infty} \mathbb{P}( z_N = k) = \lim_{N \rightarrow \infty} \mathbb{E}[f_N(k)] = (1- p)^{-\theta}\frac{\Gamma(k+\theta)}{\Gamma(k+1)\Gamma(\theta)} = \mathbb{P}( Z = k+\theta).$$
Since this is true for any $k \geq 0$, we conclude that $\displaystyle{\ell_N-m_{N-1}}\xrightarrow{D}Z$. We thus reduce the proof of the lemma to showing (\ref{S3Rat}).\\

Using (\ref{S3eq4}) we have
\begin{equation}\label{S3eq5}
\frac{f_N(k)}{f_N(0)}=\frac{\Gamma(\theta+k)\Gamma(1)}{\Gamma(1+k)\Gamma(\theta)}\prod_{i=1}^{N-2}\left(1-\frac{k}{\ell_N-m_i-\theta}\right)\prod_{i=1}^{N-1}\frac{f(\ell_N-\ell_i-\theta)}{f(\ell_N-\ell_i-k-\theta)}\mbox{, with $f(z)=\frac{\Gamma(z+1)}{\Gamma(z+\theta)}$}. 
\end{equation}
We also have from the functional equation for the gamma function $\Gamma(x+1) = x\Gamma(x)$ that
\begin{equation}\label{S3eq6}
\prod_{i=1}^{N-1}\frac{f(\ell_N-\ell_i-\theta)}{f(\ell_N-\ell_i-k-\theta)} = \prod_{i=1}^{N-1} \prod_{j = 1}^k \left( 1 - \frac{\theta - 1}{\ell_N - \ell_i - j}\right).
\end{equation}

By Lemma \ref{LSep}, we can find a sequence of sets $D(N)\subset\mathbb{R}^{2N-2}$ such that 
\begin{equation}\label{S3eq7}
\displaystyle{\lim_{N\to\infty}\mathbb{P}(\ell_1,...,\ell_N,m_1,...,m_{N-2}\in{D(N)})=1}
\end{equation}
 and for any sequence $(x_1^N,...,x_N^N,y_1^N,...,y_{N-2}^N)\in{D(N)}$ and $r \in \mathbb{R}$ we have as $N \rightarrow \infty$ that
\begin{equation*}
\mbox{1.} \hspace{1 mm}\frac{1}{x_N^N-x_{N-1}^N + r} \rightarrow 0,\hspace{2mm}
\mbox{2.} \sum_{i=1}^{N-1}\frac{1}{x_N^N-x_i^N + r} \rightarrow \frac{\log(1+\frac{1}{\sqrt{t}})}{\theta},\hspace{2mm}
\mbox{3.} \sum_{i=1}^{N-2}\frac{1}{x_N^N-y_i^N + r} \rightarrow \frac{\log(1+\frac{1}{\sqrt{t}})}{\theta}.
\end{equation*}
If $(x_1^N,...,x_N^N,y_1^N,...,y_{N-2}^N)\in{D(N)}$ is a sequence satisfying the above properties we have
\begin{equation}\label{S3eq8}
\begin{split}
&\lim_{N\to\infty}\prod_{i=1}^{N-2}\left(1-\frac{k}{x^N_N-y^N_i-\theta}\right)  \prod_{i=1}^{N-1} \prod_{j = 1}^k \left( 1 - \frac{\theta - 1}{x^N_N - x^N_i - j}\right) = \\
&\lim_{N\to\infty}\exp\left(-\sum_{i=1}^{N-2}\frac{k}{x^N_N-y^N_i-\theta}-\sum_{i=1}^{N-1}\frac{k(\theta-1)}{x^N_N - x^N_i}+o\left(\frac{1}{x^N_N-x^N_{N-1}}\right)\right) = \left( \frac{\sqrt{t}}{1+\sqrt{t}} \right)^k,
\end{split}
\end{equation}
Combining (\ref{S3eq5}), (\ref{S3eq6}), (\ref{S3eq7}) and (\ref{S3eq8}) we conclude that for every $k \geq 0$ one has
$$\displaystyle{\frac{f_N(k)}{f_N(0)} \xrightarrow{\mathbb{P}}{p^k} \frac{\Gamma(k+\theta)}{k!\Gamma(\theta)}} \mbox{ as $N \rightarrow \infty$, where $\displaystyle{p=\frac{\sqrt{t}}{1+\sqrt{t}}}$}.$$ 

It remains to show that $f_N(0)\xrightarrow{\mathbb{P}} (1-p)^{\theta}$ as $N \rightarrow \infty$. Observe that $ f_N(0) \in [0,1]$ and so it suffices to show that any weak subsequential limit of $ f_N(0)$ equals $ (1-p)^{\theta}$. Let $R$ be a subsequential limit and $f_{N_r}(0) \xrightarrow{D} R$ as $r \rightarrow \infty$. We want to show that $\mathbb{P}(R = (1-p)^\theta) = 1$.

Suppose that we have for some $\delta, \epsilon > 0$ that $\mathbb{P}(R > (1- p)^{\theta} + \epsilon) > \delta$. Then, for all $r$ large enough we would have
$$\mathbb{P}( f_{N_r} (0) >  (1- p)^{\theta} + \epsilon/2) > \delta/2.$$
The latter statement together with the fact that $\frac{f_{N_r}(k)}{f_{N_r}(0)}\xrightarrow{\mathbb{P}}{p^k} \frac{\Gamma(k+\theta)}{k!\Gamma(\theta)}$ as $r \rightarrow \infty$ implies that for $r$ large enough on an event of positive probability we will have that $\sum_{ k \geq 0 }f_{N_r}(k)> 1$, which is a contradiction. We thus conclude that $\mathbb{P}(R \leq (1- p)^{\theta} ) = 1$.

Pick any $(1 + \sqrt{t})^{-1} > \epsilon > 0 $. Using (\ref{S3eq5}), (\ref{S3eq6}) and the inequality $1 - x \leq e^{-x}$ for $x \in [0,1]$ we see that for $0 \leq k \leq \ell_N - \ell_{N-1}-\theta$
$$ \frac{f_N(k)}{f_N(0)} \leq \frac{\Gamma(\theta+k)\Gamma(1)}{\Gamma(1+k)\Gamma(\theta)}\exp\left(- \sum_{i = 1}^{N-2} \frac{k}{\ell_N - m_i -\theta} - \sum_{i = 1}^{N-1}\sum_{j = 1}^{k} \frac{\theta - 1}{\ell_N - \ell_i - j}\right) \leq $$
$$\leq \frac{\Gamma(\theta+k)\Gamma(1)}{\Gamma(1+k)\Gamma(\theta)}\exp\left(- \sum_{i = 1}^{N-2} \frac{k}{\ell_N - m_i -\theta} - \sum_{i = 1}^{N-1}\sum_{j = 1}^{k} \frac{\theta - 1}{\ell_N - \ell_i }\right),$$
where in the last inequality we used that $\theta \geq 1$. The above and (\ref{S3eq8}) imply that if $N$ is large enough and $(\ell_1,...,\ell_N,m_1,...,m_{N-2}) \in D(N)$, then for all $k\geq 0$ we have
$$f_N(0) \left(\frac{\sqrt{t}}{1 + \sqrt{t}} + \epsilon \right)^{k}\frac{\Gamma(\theta+k)}{\Gamma(\theta)\Gamma(1+k)}\geq{f_N(k)}.$$
Summing over $k$ and using the well-known identity $\displaystyle{\sum_{k = 0}^{\infty}x^{k}\frac{\Gamma(k+\theta)}{\Gamma(\theta)\Gamma(1+k)}=\left(1-x \right)^{-\theta}}$, which holds for $|x| < 1$, we get
$$ f_N(0) \left(1-\frac{\sqrt{t}}{1+\sqrt{t}}-\epsilon \right)^{-\theta}\geq{1 }, \mbox{ when $(\ell_1,...,\ell_N,m_1,...,m_{N-2}) \in D(N)$ and $N$ is large enough}.$$
 In particular, the above together with (\ref{S3eq7}) implies that $\mathbb{P}(R \geq \left(1-p-\epsilon \right)^{\theta}) = 1$. This shows that $\mathbb{P}(R \geq \left(1-p \right)^{\theta}) = 1$, which combined with our upper bound shows that $\mathbb{P}(R = (1-p)^\theta) = 1$. 
\end{proof}

\subsection{Proof of Theorem 1.1}\label{Section3.3}
We proceed by induction on $k$. The base case $k=1$, was proved in Lemma \ref{KeyLemma1}.

Suppose for $k<m$ the result of the theorem holds. We will prove it for $k=m$. For simplicity of notations we set 
$$\mbox{ $Z_j^N=X_1^{N-j+1}(tN + s;N)-X_1^{N-j}(tN + s;N)$ for $j = 1,...,k$. }$$
Take any $k_1,...,k_m \in\mathbb{Z}_{\geq{0}}$ and observe that
$$\mathbb{P}(Z_1^N=k_1,...,Z^N_{m}= k_{m})=\mathbb{P}(Z^N_{m}=k_{m}|Z^N_1=k_1,...,Z^N_{m-1}= k_{m-1})\mathbb{P}(Z^N_1=k_1,...,Z_{m-1}^N=k_{m-1}).$$
By the inductive hypothesis it is enough to show that 
\begin{equation}\label{S3spec2}
\lim_{N \rightarrow \infty} \mathbb{P}(Z^N_{m}=k_{m}|Z^N_1=k_1,...,Z^N_{m-1}=k_{m-1}) = \displaystyle{p^{k_{m}}(1-p)^{\theta}\frac{\Gamma(k_{m}+\theta)}{\Gamma(k_m+1)\Gamma(\theta)}}.
\end{equation}

Let us fix $\lambda^r \in \mathbb{Y}^r$ for $r = N-m,...,N$ such that $\lambda^{N-m} \prec \cdots \prec \lambda^N$. From Proposition \ref{S2PC} we know that 
\begin{equation}\label{S3T1}
\mathbb{P}( X^N(tN + s; N) = \lambda^N, ...,X^{N-m}(tN + s; N) = \lambda^{N-m} ) = \mathcal{J}^{multi}_{tN + s;m,N}(\lambda^{N-m},..., \lambda^N).
\end{equation}
Let $\mathcal{F}^m_N$ be the $\sigma$-algebra generated by $X_i^{N-j+1}(tN + s; N)$ for $j = 1,...,m$ and $i = 1,...,N-j+1$. Notice that $\mathcal{F}^m_N$ is a finer $\sigma$-algebra than that of $Z^N_1,...,Z^N_{m-1}$. Equations (\ref{S3T1}) and (\ref{S2mult1}) imply that for some positive $c > 0$, depending on $X^{N-m+1}(Nt + s; N)$, we have
$$\mathbb{P}( X^{N-m}(tN + s; N) = \lambda^{N-m} | \mathcal{F}^m_N) = c \cdot J_{\lambda^{N-m}}(1^{N-m})J_{X^{N-m+1}(Nt + s; N)/\lambda^{N-m}}(1).$$
In particular, we see that 
\begin{equation}\label{S3T2}
\mathbb{P}( X^{N-m}(tN + s; N) = \lambda^{N-m} | \mathcal{F}^m_N) = \mathbb{P}( X^{N-m}(tN + s; N) = \lambda^{N-m} | X^{N-m+1}(tN + s; N)).
\end{equation}

By Proposition \ref{S2PC} we know that 
\begin{equation}\label{S3T3}
\mathbb{P}( X^{N-m+1}(tN + s; N) = \lambda, X^{N-m}(tN + s; N) = \mu) = \mathcal{J}^{multi}_{tN + s;N-m,N-m+1}(\mu, \lambda).
\end{equation}
Consequently, we may apply the same arguments as in the proof of Lemma \ref{KeyLemma1} to show that 
\begin{equation}\label{S3T4}
\begin{split}
\mathbb{P}( X_1^{N-m}&(tN + s; N)  - X_1^{N-m+1}(tN + s; N) = k | X^{N-m+1}(tN + s; N))  \xrightarrow{\mathbb{P}} \\
&p^k(1- p)^{-\theta}\frac{\Gamma(k+\theta)}{\Gamma(k+1)\Gamma(\theta)}, \mbox{ for any $k \geq 0$ as $N\rightarrow \infty$.}
\end{split}
\end{equation}
The above statement follows from (\ref{S3Rat}) as well as the Tower Property and Bounded Convergence Theorem for conditional expectation.

Combining (\ref{S3T2}) and (\ref{S3T4}) we conclude that
\begin{equation}\label{S3T5}
\mathbb{P}(  X_1^{N-m}(tN + s; N)  - X_1^{N-m+1}(Nt + s; N) = k | \mathcal{F}^m_N) \xrightarrow{\mathbb{P}} p^k(1- p)^{-\theta}\frac{\Gamma(k+\theta)}{\Gamma(k+1)\Gamma(\theta)}
\end{equation}
for any $k \geq 0$ as $N\rightarrow \infty$. We now take expectation on both sides of (\ref{S3T5}) with respect to $ \mathbb{E}\left[ \cdot | Z^N_1,...,Z^N_{m-1}\right]$ and use the Tower Property for conditional expectation to get
\begin{equation}\label{S3T6}
\mathbb{P}(  Z^N_m = k | Z^N_1,...,Z^N_{m-1}) \xrightarrow{\mathbb{P}} p^k(1- p)^{-\theta}\frac{\Gamma(k+\theta)}{\Gamma(k+1)\Gamma(\theta)},
\end{equation}
for any $k \geq 0$ as $N\rightarrow \infty$. The change of the order of expectation and limits is allowed by the Bounded convergence theorem. Clearly, (\ref{S3T6}) implies (\ref{S3spec2}), which concludes the proof of the induction step. The general result now follows by induction.

\section{Dynamic limit}\label{Section4}
In this section we prove Theorem \ref{Theorem2}. The main argument is presented in Sections \ref{Section4.2} and \ref{Section4.3}. In the section below we supply several results that will be used in the proof. Throughout this section we let $D^k = D([0,\infty), \mathbb{N}_{ 0}^k)$ be the space of right-continuous paths with left limits taking values in $\mathbb{N}_{ 0}^k$ and endow it with the usual Skorohod topology (see e.g. \cite{EK}).

\subsection{Preliminaries for Theorem \ref{Theorem2}}\label{Section4.1}
In this section we consider the dynamics of the top row of $X(s;N), s\geq 0$. The main result of the section is the following.
\begin{proposition}\label{S4MainProp}
Let $X(s;N), s\geq 0$ be as in Definition \ref{mainDef} with $\theta \geq 1$ and fix $t > 0$. Then the sequence of processes $Z_N(s) = X^N_1(tN + s;N) - X_1^N(tN;N)$, $s \in [0,\infty)$ is tight on $D^1$.
\end{proposition}
The proof of Proposition \ref{S4MainProp} is given in the end of this section and relies on Lemmas \ref{S4Lemma1}, \ref{S4Lemma2} and \ref{S4Lemma3} below. We present Lemma \ref{S4Lemma1} here and postpone its proof until Section \ref{Section5}. 
\begin{lemma}\label{S4Lemma1}
Let $X(s;N), s\geq 0$ be as in Definition \ref{mainDef} with $\theta \geq 1$ and fix $s \geq 0$ and $t > 0$.  If we set $X_i = X^N_i(t;N)$ then for any $N \geq 1$ we have
\begin{equation}\label{S4ObsF}
\mathbb{E} \left[\prod_{i = 2}^N \left( 1 + \frac{\theta}{X_1 - X_i + (i-1) \theta}\right) \right] \leq \mathbb{E}\left[ \frac{X_1 +  \theta  N + 1}{t\theta}\right]
\end{equation}
If we set $\ell_i = X^N_{N-i+1}(tN + s ;N) + \theta \cdot i$ for $i = 1,...,N$ we have
\begin{equation}\label{S4ObsF2}
\prod_{j = 1}^{N-1} \left( 1 +  \frac{\theta}{\ell_N - \ell_j} \right) \xrightarrow{L^1}  \frac{\sqrt{t} + 1}{\sqrt{t} } \mbox{ as $N \rightarrow \infty$}.
\end{equation}
\end{lemma}
Lemmas \ref{S4Lemma2} and \ref{S4Lemma3} provide asymptotic statements for the process $X^N_{disc}(s), s\geq 0$. The latter process was defined in Section \ref{Section2.2} and we recall it is a continuous time Markov chain on $\mathbb{Y}^N$ with jump rates given in Proposition \ref{JumpRates}. The process implicitly depends on a parameter $\theta$, that we will assume to satisfy $\theta \geq 1$. The reason we are interested in $X^N_{disc}(s), s\geq 0$ is that by Proposition \ref{S2PC} it has the same law as top row of $X(s;N), s\geq 0$. In what follows we give two equivalent descriptions of $X^N_{disc}(s), s\geq 0$. Depending on the situation we will switch from one formulation to the other. For brevity we will write $X^N(s),s\geq 0$ for $X^N_{disc}(s), s\geq 0$.

Set $\nu = X^N$ and $x_i = \nu_i - i$ for $i = 1,...,N$, then the state space consists of ordered sequences $x_1 > x_2 > \cdots > x_N$ of integers. Each particle $x_i$ jumps to the right by $1$ independently of the others according to an exponential clock with rate $\lambda_i = q_{\nu \rightarrow \nu \sqcup (i, \nu_i + 1)}$ if $\nu \sqcup (i, \nu_i + 1) \in \mathbb{Y}^N$ and $0$ otherwise (these rates are given in (\ref{eqJR})). In particular, the jump rate is non-zero only if the position to the right of a particle is unoccupied. We remark that the above particle dynamics has global interactions, as the jump rate of each particle is influenced by the position of all other particles.

The second dynamics we formulate is a consequence of Proposition \ref{poisson}, which states that for any $\nu \in \mathbb{Y}^N$,
$ \sum_{\mu = \nu \sqcup \square} q_{\nu \rightarrow \mu} = N \theta.$ The latter implies that if $\lambda_i = q_{\nu \rightarrow \nu \sqcup (i, \nu_i + 1)}$ for $\nu \sqcup (i, \nu_i + 1) \in \mathbb{Y}^N$ and $0$ otherwise, then $p_i = \lambda_i/ N\theta$ defines a probability distribution $\mathbb{P}^{\nu}$ on $\{1,...,N\}$. Let $\nu(n) = Y_1(n) \geq Y_2 \geq \cdots \geq Y_N(n)$, $n \geq 0$ be the discrete time Markov chain on $\mathbb{Y}^N$, where at each time $n$ we sample $i$ from $\{1,...,N\}$ according to $\mathbb{P}^{\nu(n)}$ and increase $Y_i$ by $1$. If $M_s$ is a Poisson process on $\mathbb{R}_+$ with intensity $N\theta$, which is independent of $\nu(n)$, $n \geq 0$, and $\nu(0) = \varnothing$, then one readily observes that the process $\nu(M_s) = Y_1(M_s) \geq Y_2(M_s) \geq \cdots \geq Y_N(M_s)$, $s \geq 0$ has the same law as  $X^N(s), s\geq 0.$\\

We now state and prove Lemmas \ref{S4Lemma2} and \ref{S4Lemma3}.
\begin{lemma}\label{S4Lemma2}
Fix $T_0, T_1 > 0$, $\theta \geq 1$ and let $X^N(s), s\geq 0$ be as above. We can find a constant $C > 0$ depending on $\theta$, $T_0$ and $T_1$, such that for any $\Delta \in [0,1]$ and $s \in [0, T_1]$ we have
\begin{equation}
\limsup_{N \rightarrow \infty} \mathbb{E} \left[ X^N_1(T_0N + s + \Delta) - X_1^N(T_0N + s) \right] \leq C\Delta.
\end{equation}
\end{lemma}
\begin{proof}
Let $\Delta \in [0,1]$,  and $s \in [0,T_1]$ be given and set $\epsilon = \Delta/ N^3$. Denote $M_t = X^N_1(t) + \cdots + X_N^N(t)$. For $r = 1,...,N^3$ we define
$$t_r= s + r\epsilon, \hspace{5mm} A_r=  X^N_1(T_0N + t_r) - X_1^N(T_0N + t_{r-1}), \hspace{5mm}B_r =  M_{T_0N + t_r} -  M_{T_0N + t_{r-1}}.$$ 
 Let $\lambda_1(t)$ denote the jump rate of the rightmost particle $X^N_1(t)$, which by (\ref{eqJR}) equals
\begin{equation}\label{S4eq2}
\lambda_1 (t) = \theta \cdot \prod_{i = 2}^N \left( 1 + \frac{\theta}{X^N_1(t) - X^N_i(t) + \theta(i - 1)}\right).
\end{equation}
In view of our second dynamic formulation (see the discussion before the statement of the lemma) we have
$$\mathbb{E}[A_r] = \mathbb{E}[A_r | B_r= 0]\cdot\mathbb{P}(B_r = 0)  + \mathbb{E}[A_r | B_r = 1]\cdot\mathbb{P}(B_r = 1) + \mathbb{E}[A_r | B_r  \geq 2]\cdot\mathbb{P}(B_r \geq 2) $$
$$ = \mathbb{E}\left[\frac{\lambda_1(t_{r-1})}{N \theta} \right]\cdot \mathbb{P}(B_r = 1) +\mathbb{E}[A_r | B_r  \geq 2]\cdot\mathbb{P}(B_r \geq 2).$$
Since $B_r$ is a Poisson random variable with parameter $\theta N \epsilon = \Delta/N^2 \leq 1/N^2$ and $A_r \leq B_r$, we have that
\begin{itemize}
\item $\mathbb{P}(B_r = 1)  = \exp\left( -\frac{\Delta}{N^2}\right) \cdot \frac{\Delta}{N^2} \leq \frac{\Delta}{N^2} \mbox{ and }$
\item $\mathbb{E}[A_r | B_r  \geq 2]\cdot\mathbb{P}(B_r \geq 2) \leq  \mathbb{E}[B_r | B_r  \geq 2]\cdot\mathbb{P}(B_r \geq 2) = \mathbb{E}[B_r] - \mathbb{P}(B_r = 1) = \frac{\Delta}{N^2}\cdot \left( 1 - \exp\left( - \frac{\Delta}{N^2}\right)\right) \leq \frac{2}{N^4}$, with the latter inequality true for all large $N$.
\end{itemize}
The above inequalities show that 
\begin{equation}\label{S4L2eq1}
\mathbb{E}[A_r] \leq \frac{\Delta}{N^3 \theta} \cdot  \mathbb{E}\left[\lambda_1(t_{r-1}) \right] + \frac{2}{N^4}.
\end{equation}

From Proposition \ref{S2PC} we know that $X^N(s;N)$ and $X^N(s)$ have the same law. Consequently, we may apply (\ref{S4ObsF}) to conclude
\begin{equation}\label{S4L2eq2}
 \mathbb{E}\left[\lambda_1(t_{r-1}) \right] \leq \theta \cdot \mathbb{E}\left[\frac{X_1^N(T_0N + t_{r-1}) + \theta N + 1}{t_{r-1} \theta }\right] \leq  \theta \cdot \mathbb{E}\left[\frac{X_1^N(T_0N + T_1 + 1) + \theta N + 1}{T_0 N \theta }\right].
\end{equation}
In the second inequality we used that $X_1^N(s), s\geq 0$ is an increasing process and that $t_{r-1} \in [T_0N, T_0 + T_1 + 1]$.

It follows from (\ref{S4L2eq1}) and (\ref{S4L2eq2}) that
$$\limsup_{N\rightarrow \infty} \mathbb{E} \left[ X^N_1(T_0N + s + \Delta;N) - X_1^N(T_0N + s;N) \right] = \limsup_{N\rightarrow \infty} \sum_{r = 1}^{N^3} \mathbb{E}[A_r] \leq $$
$$\leq  \limsup_{N\rightarrow \infty} \sum_{r = 1}^{N^3} \left(  \frac{\Delta}{N^3 } \cdot \mathbb{E}\left[\frac{X_1^N(T_0N + T_1 + 1) + \theta N}{T_0 N \theta }\right] + \frac{2}{N^4} \right) = \Delta \frac{(1 + \sqrt{T_0})^2}{T_0}.$$
In obtaining the last equality we used that $X^N(s;N)$ and $X^N(s)$ have the same law by Proposition \ref{S2PC} and (\ref{S3Right}).
\end{proof}

\begin{lemma}\label{S4Lemma3}
Fix $T_0, T_1 > 0$, $\theta \geq 1$ and let $X^N(s), s\geq 0$ be as above. Define
\begin{equation}\label{S4eq3}
\lambda_1(s) = \theta \cdot \prod_{i = 2}^N \left( 1 + \frac{\theta}{X^N_1(s) - X^N_i(s) + \theta(i - 1)}\right).
\end{equation}
If $M > \theta\cdot \frac{1 + \sqrt{T_0}}{\sqrt{T_0}}$ then 
\begin{equation}\label{S4eq10}
\lim_{N \rightarrow \infty} \mathbb{P} \left( \lambda_1( s) < M \mbox{ for } s \in [T_0N, T_0N + T_1] \right) = 1.
\end{equation}
\end{lemma}
\begin{proof}
Let $\epsilon > 0$ be given. We know we have the following convergence statements
\begin{equation}\label{S4L3eq1}
\begin{split}
&\mbox{1. } \prod_{i = 2}^N \left( 1 + \frac{\theta}{X^N_1(s) - X^N_i(s) + \theta(i - 1) - K}\right)\xrightarrow{\mathbb{P}}   \frac{1 + \sqrt{T_0}}{\sqrt{T_0}} \mbox{ as $N \rightarrow \infty$, for fixed $s,K \geq 0$;}\\
&\mbox{2. }X^N_1(T_0N) -X^N_2(T_0N +T_1)  \xrightarrow{\mathbb{P}} \infty \mbox{ as $N \rightarrow \infty$};\\
&\mbox{3. }\mbox{ There exists a $K_0 \in \mathbb{N}$, such that } \limsup_{N \rightarrow \infty} \mathbb{P}( X^N_1(T_0N + T_1) -X^N_1(T_0N) \geq K_0) < \epsilon/2. \\
\end{split}
\end{equation}
From Proposition \ref{S2PC} we know that $X^N(s;N)$ and $X^N(s)$ have the same law. Consequently, the first statement above follow from Lemma \ref{LSep} and the inequalities $e^{-x - x^2} \leq 1 - x \leq e^{-x + x^2}$, which hold for $x \in [0, 0.2]$. The second statement follows from Lemmas \ref{LSep} and \ref{S4Lemma2}.  The final statement is a consequence of Chebyshev's inequality and Lemma \ref{S4Lemma2}.\\

Fix the event $A_N = \left \{X^N_1(T_0N ) > 2K_0 + X^N_2(T_0N + T_1),X^N_1(T_0N + T_1) -X^N_1(T_0N) \leq K_0  \right\}$.
Since $X^N_i(T_0N + s)$ increases in $s$, we see that on $A_N$ we have for $s \in [0, T_1]$ and $i \in \{2,...,N\}$
$$   \left( 1 + \frac{\theta}{X^N_1(T_0N + s) -X^N_i(T_0N + s) + \theta(i - 1) }\right) \leq \left( 1 + \frac{\theta}{X^N_1(T_0N ) -X^N_i(T_0N + T_1) + \theta(i - 1)  }\right) $$
$$ \leq    \left( 1 + \frac{\theta}{X^N_1(T_0N + T_1 ) -X^N_i(T_0N + T_1) + \theta(i - 1)  - K_0  }\right). $$
Taking the product over $i = \{2,...,N\}$ above we conclude that on $A_N$
$$\lambda_1(s)  \leq \theta \cdot \prod_{i = 2}^N\left( 1 + \frac{\theta}{X^N_1(T_0N + T_1 ) -X^N_i(T_0N + T_1) + \theta(i - 1)  - K_0  }\right).$$
From 1. in (\ref{S4L3eq1}) we have that the quantity on the right above converges to $\theta \cdot \frac{1 + \sqrt{T_0}}{\sqrt{T_0}}$ in probability, which is less than $M$. We thus conclude that
$$\liminf_{N \rightarrow \infty} \mathbb{P} \left( \lambda_1(s) < M \mbox{ for } s \in [T_0N, T_0N + T_1] \right) \geq \liminf_{N \rightarrow \infty} \left( 1 - \mathbb{P}(A_N^c) \right) = 1- \limsup_{N\rightarrow \infty}\mathbb{P}(A_N^c) .$$
It follows from 2. and 3. in (\ref{S4L3eq1}) that $\limsup_{N\rightarrow \infty}\mathbb{P}(A_N^c)  \leq \epsilon/2$ and so we conclude that 
$$\liminf_{N \rightarrow \infty} \mathbb{P} \left( \lambda_1( s) < M \mbox{ for } s \in [T_0N, T_0N + T_1] \right) > 1- \epsilon.$$
As $\epsilon > 0$ was arbitrary the statement of the lemma follows.
\end{proof}

\begin{proof}(Proposition \ref{S4MainProp})
We verify the necessary and sufficient conditions for tightness from Corollary 3.7.4 in \cite{EK}. Firstly, we note that for any $s \geq 0$, $Z_N(s)$ are tight on $\mathbb{R}$ because $Z_N(s) \geq 0$ and by Lemma \ref{S4Lemma2}, the expectations of these variables are uniformly bounded by a constant. The latter verifies the first condition of Corollary 3.7.4 in \cite{EK}.

Because $Z_N(s)$ is a counting process (it is increasing, pure-jump and has unit jump sizes) the second condition reduces to showing that for any $\eta > 0$ and $T > 0$ there exists a $\delta > 0$ such that
\begin{equation}\label{S4eq8}
\limsup_{N \rightarrow \infty} \mathbb{P} \left( \min_{i} \left[  T^N_{i+1} - T^N_i \right] \leq \delta \right) \leq \eta, 
\end{equation}
where $T^N_1 < T^N_2 <...$ are the jump times of $Z_N$ in $[0,T]$. Informally, the meaning of (\ref{S4eq8}) is that on any compact inverval $[0,T]$ the jump times of $Z_N$ are well-separated with high probability. The reason one expects the jump times of $Z_N$ to be well-separated is that the jump rate for this process at time $s$ has the same law as $\lambda_1(tN + s)$, which is given in (\ref{S4eq3}), and the latter quantity behaves like a constant for all large $N$.\\

In what follows we will construct a Poisson point process $R_N$, which is coupled with $Z_N$, and with high probability contains $T^N_1,T^N_2,...$ as a subset of its own jump times in $[0,T]$. We start by fixing $M <  N\theta$ and considering the process $\nu(M_s) = Y_1(M_s) \geq Y_2(M_s) \geq \cdots \geq Y_N(M_s), s\geq 0$ that was discussed before Lemma \ref{S4Lemma2}. Let $S_1 < S_2 < ...$ be the arrival times of $M_s$ in the interval $[tN, tN + T]$, which we visualize as points on this segment. We now follow these points from left to right and color some of them in red as follows.

We start from $S_1$ and look at $\nu(M_{S_1} - 1)$. Let $\lambda_1(S_1)$ be given by
$$\theta \cdot \prod_{i = 2}^N \left( 1 + \frac{\theta}{Y_1(M_{S_1} - 1)- Y_i(M_{S_1} - 1) + \theta(i - 1)}\right)$$
 and suppose it is less than $M$. Then we color $S_1$ in red if $Y_1(M_{S_1}) - Y_1(M_{S_1} - 1) = 1$. If the latter is not true then we still color the point in red with probability $\frac{M - \lambda_1}{N\theta  - \lambda_1}$. Since $Y_1$ jumps at time $S_1$ precisely with probability $\frac{\lambda_1}{N\theta}$ we conclude that this way we colored $S_1$ in red with probability $M/N\theta$. Afterwards we continue in this fashion until we reach the end of the interval $[tN, tN + T]$ or until we reach some $S_i$ such that $\lambda_1(S_i) > M$. When the latter happens we simply color the point $S_i$ in red with probability $M/N\theta$. Overall, the red points in the interval $[tN, tN + T]$ were obtained by coloring each of the arrival times of $M_s$ independently with probability $M/N\theta$. Thus if $R_N$ denotes the point process on $[0,T]$, which is obtained by shifting the red points to the left by $tN$, we conclude that $R_N$ is a Poisson point process with parameter $M$ (recall that $M_s$ is a Poisson point process with parameter $N\theta$).

Let $U^N_1,U^N_2,...$ be the arrival times for $R_N$ in $[0,T]$. By construction, we know that on the event $E_N^M = \{ \lambda_1(s) < M| s\in  [tN, tN + T] \}$ the set $U^N_1,U^N_2,...$ contains $T^N_1,T^N_2,...$ as a subset. The latter implies that on the event $E^M_N$ we have $\min_{i} \left[  T^N_{i+1} - T^N_i \right] \leq \min_{i} \left[  U^N_{i+1} - U^N_i \right]$ and so we conclude that 
\begin{equation}\label{S4eq9}
\mathbb{P} \left( \min_{i} \left[  T^N_{i+1} - T^N_i \right] \leq \delta \right) \leq \mathbb{P} \left( \min_{i} \left[  U^N_{i+1} - U^N_i \right] \leq \delta \right) + \mathbb{P} \left( (E^M_N)^c \right)
\end{equation}
Fix $M >  \theta \cdot \frac{1 + \sqrt{t}}{\sqrt{t}}$ and notice that as $R_N$ is a Poisson point process with parameter $M$, we can find $\delta > 0$ such that
$$\mathbb{P} \left( \min_{i} \left[  U^N_{i+1} - U^N_i \right] \leq \delta \right)  \leq \eta/2, \mbox{ for all $N \in \mathbb{R}$}.$$
On the other hand we have by Lemma \ref{S4Lemma3} that $\mathbb{P} \left( (E^M_N)^c \right) \rightarrow 0$ as $N \rightarrow \infty$. Combining these estimates with (\ref{S4eq9}) we conclude (\ref{S4eq8}). This proves that $Z_N(s), s\geq 0$ is tight on $D^1$.
\end{proof}

\subsection{One row analysis}\label{Section4.2}
In this section we focus on the top row of $X(tN + s;N), s\geq 0$ and analyze the limiting distribution of the rightmost particle $X^N_1(tN + s;N), s\geq 0$. The main result we will prove is the following.
\begin{proposition}\label{S4Prop1}
Let $X(s;N), s\geq 0$ be as in Definition \ref{mainDef} and fix $t > 0$. Then the sequence of processes $Z_N(s) = X^N_1(tN + s;N) - X_1^N(tN;N)$, $s \in [0,\infty)$ converges in the limit $N \rightarrow \infty$ in law on $D^1$ to the Poisson point process with rate $ q = \theta \cdot \frac{1 + \sqrt{t}}{\sqrt{t}}$.
\end{proposition}
\begin{proof}
From Proposition \ref{S2PC} we know that $X^N(s;N)$ and $X^N(s)$ (defined in Section \ref{Section4.1}) have the same law and so it suffices to prove the proposition for $X^N(s)$. In particular, we let $Z_N(s) = X^N_1(tN + s) - X_1^N(tN)$, $s \in [0,\infty)$ and prove that the sequence converges in the limit $N \rightarrow \infty$ in law on $D^1$ to the Poisson point process with rate $ q = \theta \cdot \frac{1 + \sqrt{t}}{\sqrt{t}}$.

From Proposition \ref{S4MainProp} we know that $Z_N$ is a tight family on $D^1$. Let $Z$ be any subsequential limit and pick a subsequence $Z_{N_k}$, which converges in law to $Z$. By virtue of the Skorohod Embedding Theorem (see e.g. Theorem 3.5.1 in \cite{Du}) we may assume that all processes involved are defined on the same probability space and that the convergence holds in the almost sure sense. Our goal is to show that $Z$ is the Poisson point process with rate $q$. 

The strategy is to use the Martingale Problem, which characterizes the Poisson process with rate $q$ as the unique process $R_u$ such that $R_0 = 0$ and for every bounded function $f: \mathbb{N}_{ 0} \rightarrow \mathbb{R}$, we have that 
\begin{equation}\label{S4eq12}
M(u): = f(R_u) - \int_0^u q( f(R_s+1) - f(R_s))ds
\end{equation}
is an $\mathcal{F}^R_u = \sigma( R_s, s\in [0,u])$ martingale. The latter result is a special case of Theorem 4.4.1 in \cite{EK}. By uniqueness we mean that if two processes $R^1_u$ and $R^2_u$ with sample paths in $D^1$ satisfy the above condition then they have the same finite dimensional distributions. The latter by Proposition 3.7.1 in \cite{EK} means that the two processes define the same law on $D^1$. Since the Poisson process of rate $q$ clearly satisfies (\ref{S4eq12}), we conclude that it suffices to show that for any bounded function $f: \mathbb{N}_{ 0} \rightarrow \mathbb{R}$ we have that
\begin{equation}\label{eq13}
M(u): = f(Z(u)) - \int_0^u q( f(Z(s)+1) - f(Z(s)))ds
\end{equation}
is an $\mathcal{F}^Z_u$ martingale and $Z(0) = 0$ a.s. The second condition is immediate from $Z_N(0) = 0$ for each $N$ by definition and $Z_N(0) \rightarrow Z(0)$ a.s. by assumption. Since $Z(u)$ is right-continuous and $f$ is bounded we see that $M(u)$ is adapted to $\mathcal{F}^Z_u$ and integrable. The only thing left to check is that for $u \geq s$ one has
\begin{equation}\label{S4eq14}
\mathbb{E}\left[ M(u) - M(s) | \mathcal{F}^Z_s\right] = 0 \iff \mathbb{E}\left[ {\bf 1}_A \cdot (M(u) - M(s))\right] = 0 \mbox{ for all $A \in \mathcal{F}^Z_s$}.
\end{equation}
The collection of sets $A$ that satisfy (\ref{S4eq14}) is a $\lambda$-system, and so if we can prove that (\ref{S4eq14}) holds for sets of the form $A = \{ Z(s_1) \leq a_1,..., Z(s_k) \leq a_k\}$ where $k \in \mathbb{N}$, $a_i \in \mathbb{R}$ and $0 \leq s_1 \leq s_2 \leq ... \leq s_k \leq s$, then by the $\pi-\lambda$ Theorem we will have the statement for all sets $A \in \mathcal{F}^Z_s$. We conclude that what remains to be proved is
\begin{equation}\label{S4eq15}
\mathbb{E}\left[ {\bf 1}_A \cdot (M(u) - M(s))\right] = 0 \mbox{ if $A = \{ Z(s_1) \leq a_1,..., Z(s_k) \leq a_k\}$ and $u \geq s$}.
\end{equation}

Let us introduce the following notation
\begin{equation}
\begin{split}
&A_N = \{ Z_N(s_1) \leq a_1,..., Z_N(s_k) \leq a_k\}, \\
&M_N(u) =  f(Z_N(u)) - \int_0^u \lambda_1(tN + v)( f(Z_N(v)+1) - f(Z_N(v)))dv,\\
&M_N'(u) = f(Z_N(u)) - \int_0^u q( f(Z_N(v)+1) - f(Z_N(v)))dv.
\end{split}
\end{equation}
In the above we have that $\lambda_1(s)$ is given by (\ref{S4eq3}) and is the jump rate of the particle $X^N_1(s), s\geq 0$.
The Martingale Problem for $X^N(s), s\geq 0$ shows that $M_N(u)$ is a martingale with respect to the filtration $\mathcal{F}^{X^N}_{u + tN}$. In particular, we conclude that 
\begin{equation}\label{S4Prop1eq1}
\mathbb{E}\left[ {\bf 1}_{A_N} \cdot (M_N(u) - M_N(s))\right] = 0, \mbox{ when $u \geq s$.}
\end{equation}
By the Bounded Convergence Theorem we have 
\begin{equation}\label{S4Prop1eq2}
 \lim_{k \rightarrow \infty} \mathbb{E} \left[ {\bf 1}_{A_{N_k}} \cdot (M'_{N_k}(u) - M'_{N_k}(s)) - {\bf 1}_A \cdot (M(u) - M(s)) \right] = 0.
\end{equation}
Combining (\ref{S4Prop1eq1}) and (\ref{S4Prop1eq2}) we reduce (\ref{S4eq15}) to showing the following statement for $u \geq s$
\begin{equation}\label{S4Prop1eq3}
\lim_{N \rightarrow \infty} \mathbb{E} \left[ {\bf 1}_{A_N} \cdot K_N(s,u)\right] = 0, \mbox{ where } K_N(s,u) = M'_N(u) - M'_N(s) - M_N(u) + M_N(s).
\end{equation}

We notice that
$$K_N(s,u) = \int_s^u (\lambda_1(tN + v) - q)\cdot( f(Z_N(v)+1) - f(Z_N(v)))dv.$$
Let $F = \sup_{x \in \mathbb{N}_{ 0}} |f(x)|$. Then we have that 
$$\mathbb{E}\left [ | K_N(s,u)| \right] \leq 2F\int_{s}^u  \mathbb{E}\left[| \lambda_1(tN + v) - q | \right]dv.$$
From (\ref{S4ObsF2}) we know that for each $v \geq 0$ we have $\lim_{N \rightarrow \infty}\mathbb{E}\left[| \lambda_1(tN + v) - q | \right] = 0$. On the other hand, we have for $v \in [0,u]$ that
\begin{equation}
\mathbb{E}\left[| \lambda_1(tN + v) - q | \right] \leq q + \mathbb{E}\left[ \lambda_1(tN + v) \right] \leq q +  \mathbb{E}\left[ \frac{X_1^N(tN + u) + \theta N}{\theta t N} \right] \leq C.
\end{equation}
The middle inequality follows from the fact that $X^N(s;N)$ and $X^N(s)$ have the same law by Proposition \ref{S2PC}, coupled with (\ref{S4ObsF}) and the monotonicity of $X_1^N(tN + v)$ for $v \in [0,u]$. The last inequality is a consequence of (\ref{S3Right}). An application of the Bounded Convergence Theorem now reveals that $K_N(s,u) \xrightarrow{L^1} 0$ as $N \rightarrow \infty$. The latter implies equation (\ref{S4Prop1eq3}) and hence the proposition.
\end{proof}

\subsection{Proof of Theorem \ref{Theorem2}}\label{Section4.3}
By Proposition \ref{S2PC} we know that the projection of $X(s;N), s\geq 0$ to the top $k+1$ levels has the same law as $X^{multi}_{N- k,N}(s), s\geq 0$ from Definition \ref{mainDef2}. Consequently, it is enough to prove the theorem for this process. For brevity we denote  $X^{multi}_{N- k,N}(s), s\geq 0$ by $X^j_i(s), s\geq 0$ with $j = N-k,...,N$ and $1 \leq i \leq j$. 
Define the sequence of processes 
\begin{equation}\label{S4eq11}
\begin{split}
Q^N(s) = &\left( Q_1^N(s),...,Q_k^N(s) \right)  = ( X^N_1(tN + s) - X_1^{N-1}(tN + s),X^{N-1}_1(tN + s) -\\ &X_1^{N-2}(tN + s)
,...,X^{N-k+1}_1(tN + s) - X_1^{N-k}(tN+ s)),  \hspace{2mm}s \geq 0.
\end{split}
\end{equation}
 To prove the theorem we want to show that $Q^N(s) $ converges in the limit $N \rightarrow \infty$ in law on $D^k$ to the process $Q(s)$ from Definition \ref{S1DefQ}.\\

We start by showing that $Q^N(s), s\geq 0$ is tight on $D^k$. It suffices to show that for each $i \in \{1,...,k\}$ we have that $Q^N_i(s)$  is tight in $D^1$. Notice that 
$$Q^N(s) = \left(X^{N-i+1}_1(tN + s)  - X^{N-i+1}_1(tN) \right) - $$
$$ - \left(X^{N-i}_1(tN + s)  - X^{N-i}_1(tN) \right) + \left( X^{N-i+1}_1(tN)  - X^{N-i}_1(tN)\right).$$
The first two summands are tight on $D^1$ by Proposition \ref{S4MainProp}, while the last summand is tight by Theorem \ref{Theorem1}. We conclude that $Q^N$ is tight on $D^k$. 

Our strategy for the remainder is to use the Martingale Problem, similarly to our proof of Proposition \ref{S4Prop1}. For a $k$-tuple  $x = (x_1,...,x_k) \in \mathbb{N}_{ 0}^k$ we let
\begin{equation}
\lambda_{ij}(x) =  {\bf 1}_{\{x_{j} > 0\}}  \prod_{r = j + 1}^{i-1} {\bf 1}_{\{x_r = 0\}} \times \theta \cdot \frac{\theta + x_i}{1 + x_i} \mbox{ for $0 \leq j < i \leq k$, and } 
\end{equation}
\begin{equation}
\lambda_{k+1, j}(x) = {\bf 1}_{\{x_{j} > 0\}}  \prod_{r = j + 1}^{i-1} {\bf 1}_{\{x_r = 0\}} \times  \theta \cdot \frac{1 + \sqrt{t}}{\sqrt{t}} \mbox{ for $0 \leq j \leq k$}.
\end{equation}
with the convention that $x_0 > 0$. We also let ${\bf e}_i$ denote the $i$-th standard vector in $\mathbb{R}^k$ and write ${\bf e}_0 = {\bf e}_{k+1}$ for the zero vector.

By definition, the Markov process $Q(s), s \geq 0$ solves the following Martingale Problem. Let $f: \mathbb{N}^{k}_{ 0} \rightarrow \mathbb{R}$ be a bounded function. Then the process
\begin{equation}\label{S4eq22}
M_Q(u): = f(Q(u)) - \int_0^u \sum_{0 \leq j < i \leq k + 1} \lambda_{ij}(Q(v)) \cdot \left[ f(Q(v) + {\bf e}_i - {\bf e}_j) - f(Q(v))\right] dv
\end{equation}
is an $\mathcal{F}^Q_u$ martingale. It follows from Theorem 4.4.1 in \cite{EK} that if $R(u)$ is another process with sample paths in $D^k$, which solves the above Martingale Problem and $R(0)$ has the same distribution as $Q(0)$ then $R$ and $Q$ have the same finite-dimensional distributions.

From our earlier work we know that $Q^N$ form a tight family on $D^k$. Let $R$ be any subsequential limit and pick a sequence $Q_{N_m}$, which converges in law to $R$. By the Skorohod Embedding Theorem (see e.g. Theorem 3.5.1 in \cite{Du}) we may assume that all processes involved are defined on the same probability space and that the convergence holds in the almost sure sense. In addition, from Theorem \ref{Theorem1} we know that $R(0)$ has the same distribution as $Q(0)$. What remains to be shown is that $R(u)$ satisfies the Martingale Problem of (\ref{S4eq22}). 

Let us fix a bounded function $f: \mathbb{N}^{k}_{ 0} \rightarrow \mathbb{R}$ and let $M_R(u)$ be the process of (\ref{S4eq22}) with $Q$ replaced with $R$. Similarly to the proof of Proposition \ref{S4Prop1} we reduce the proof that $M_R(u)$ is a martingale to showing that 
\begin{equation}\label{S4eq23}
\mathbb{E}\left[ {\bf 1}_A \cdot (M_R(u) - M_R(s))\right] = 0 \mbox{ if $A = \{ R(s_1) \in B_1,..., R(s_l) \in B_l\}$},
\end{equation}
where $s_1 \leq s_2 \leq ... \leq s_l\leq s \leq u$ and $B_i \in  \mathcal{B}(\mathbb{R}^k)$ (the Borel $\sigma$-algebra on $\mathbb{R}^k$). 

We introduce the following notation
\begin{equation}
\begin{split}
&A_N = \{ Q_N(s_1) \in B_1,..., Q_N(s_l) \in B_l\}, \\
&M_N(u) =  f(Q^N(u)) - \int_0^u \sum_{0 \leq j < i \leq k + 1} \hspace{-5mm} \lambda^N_{ij}(X(v + tN)) \left[f(Q^N(v) + {\bf e}_i - {\bf e}_j) - f(Q^N(v))\right]dv\\
&M_N'(u) = f(Q^N(u)) - \int_0^u\sum_{0 \leq j < i \leq k + 1}  \hspace{-5mm}\lambda_{ij}(Q^N(v))  \left[ f(Q^N(v) + {\bf e}_i - {\bf e}_j) - f(Q_N(v))\right]dv.
\end{split}
\end{equation}
In the above formula we have
\begin{equation*}
\lambda^N_{\alpha\beta}(\lambda^{N-k},...,\lambda^N) =  \theta\cdot {\bf 1}_{\{\lambda_1^{N-\beta+1}  > \lambda_1^{N-\alpha+1}\}} \hspace{-1mm}  \prod_{r = \beta + 1}^{\alpha-1} \hspace{-1mm} {\bf 1}_{\{\lambda_1^{N-r+1}  = \lambda_1^{N-\alpha+1}\}}   \prod_{j = 2}^{N}  \frac{ 1 - \frac{ \theta - 1}{\lambda^{N-\alpha+1}_1 - \lambda^{N-\alpha+1}_j + (j-1)\theta}}{1 - \frac{ \theta - 1}{\lambda^{N-\alpha+1}_1 - \lambda^{N-\alpha}_{j-1} + (j-1)\theta }} 
\end{equation*}
for $0 \leq \beta < \alpha \leq k+1$ and $(\lambda^N,...,\lambda^{N-k}) \in \mathbb{GT}^{N}_{N-k}$ . We also set
\begin{equation*}
\begin{split}
\lambda^N_{k+1, \beta}(\lambda^{N-k},...,\lambda^N) =\theta \cdot {\bf 1}_{\{\lambda_1^{N-\beta+1}  > \lambda_1^{N-k}\}}  \prod_{r = \beta + 1}^{k} {\bf 1}_{\{\lambda_1^{N-r+1}  = \lambda_1^{N-k}\}}  \times\\
 \prod_{j = 2}^{N-k} \left( 1 + \frac{\theta}{\lambda^{N-k}_1 - \lambda^{N-k}_j + (\theta - 1)(j - 1)}\right),
\end{split}
\end{equation*}
for $\beta = 0,...,k$. In both equations we use the convention that $\lambda^{N+1}_1 = \infty$. The meaining of $\lambda_{ij}^N$ is that it equals the jump rate with which particles $X^{N-i+1}_1,...,X_1^{N-j+1}$ (and no others) jump to the right by $1$. The formulas presented above are obtained from (\ref{Q2}) and Proposition \ref{S2Prop3}. 

The Martingale Problem for $X(s), s\geq 0$ shows that $M_N(u)$ is a martingale with respect to the filtration $\mathcal{F}^{X}_{u + NT_0}$. In particular, we conclude that 
\begin{equation}\label{S4eq24}
\mathbb{E}\left[ {\bf 1}_{A_N} \cdot (M_N(u) - M_N(s))\right] = 0 \mbox{ when $u \geq s$}.
\end{equation}
We notice that by the Bounded Convergence Theorem we have 
\begin{equation}\label{S4eq25}
 \lim_{m \rightarrow \infty} \mathbb{E} \left[ {\bf 1}_{A_{N_m}} \cdot (M'_{N_m}(u) - M'_{N_m}(s)) - {\bf 1}_A \cdot (M(u) - M(s)) \right] = 0.
\end{equation}
Combining (\ref{S4eq24}) and (\ref{S4eq25}) we reduce (\ref{S4eq23}) to showing the following statement for $u \geq s$
\begin{equation}\label{S4eq26}
\lim_{N \rightarrow \infty} \mathbb{E} \left[ {\bf 1}_{A_N} \cdot K_N(s,u)\right] = 0, \mbox{ where } K_N(s,u) = M'_N(u) - M'_N(s) - M_N(u) + M_N(s).
\end{equation}

We notice that
$$K_N(s,u) = \sum_{0 \leq j < i \leq k + 1} \int_s^u(\lambda^N_{ij}(X(v + tN)) -\lambda_{ij}(Q^N(v))   \cdot ( f(Q^N(v) + {\bf e}_i - {\bf e}_j) - f(Q^N(v))))dv.$$
Let $F = \sup_{x \in \mathbb{N}_{ 0}} |f(x)|$. Then we have that 
\begin{equation}\label{S4eq27}
\mathbb{E}\left [ | K_N(s,u)| \right] \leq \sum_{0 \leq j < i \leq k + 1} 2F\int_{s}^u  \mathbb{E}\left[| \lambda^N_{ij}(X(v + tN)) -\lambda_{ij}(Q^N(v))  | \right]dv.
\end{equation}

From Proposition \ref{S2PC} and (\ref{S4ObsF2}) we know that for each $u \geq 0$ we have 
$$\lim_{N \rightarrow \infty}\mathbb{E}\left[ \left| \lambda^N_{k+1,j}(X(u + tN)) -\lambda_{k+1,j}(Q^N(u))  \right| \right]= 0.$$ 
On the other hand we have
\begin{equation*}
\mathbb{E}\left[| \lambda^N_{k+1,j}(X(u + tN)) -\lambda_{k+1,j}(Q^N(u)) | \right] \leq \mathbb{E}\left[ \lambda^N_{k+1,j}(X(u + tN))  \right] + \theta \cdot \frac{1 + \sqrt{t}}{\sqrt{t}}  \leq C.
\end{equation*}
The middle inequality follows from Proposition \ref{S2PC}, coupled with (\ref{S4ObsF}) and the monotonicity of $X_1^{N-k}(tN + v)$ for $v \in [0,u]$. The last inequality is a consequence of (\ref{S3Right}). An application of the Bounded Convergence Theorem now reveals that 
\begin{equation}\label{S4eq28}
\lim_{N \rightarrow \infty} \sum_{0 \leq j  \leq  k } 2F\int_{s}^u  \mathbb{E}\left[| \lambda^N_{k+1, j}(X(v + tN)) -\lambda_{k+1,j}(Q^N(v))  | \right]dv = 0.
\end{equation}
In addition, by combining Proposition \ref{S2PC} and \ref{LSep}, we know that $\lambda^N_{ij}(X(v + tN)) -\lambda_{ij}(Q^N(v)) \xrightarrow{\mathbb{P}} 0$. This together with the Bounded Convergence Theorem shows
\begin{equation}\label{S4eq29}
\lim_{N \rightarrow \infty} \sum_{0 \leq j < i \leq k } 2F\int_{s}^u  \mathbb{E}\left[| \lambda^N_{ij}(X(v + tN)) -\lambda_{ij}(Q^N(v))  | \right]dv.
\end{equation}
Combining (\ref{S4eq27}), (\ref{S4eq28}) and (\ref{S4eq29}) shows that $K_N(s,u) \xrightarrow{L^1} 0$ as $N \rightarrow \infty$. The latter implies equation (\ref{S4eq26}) and hence the theorem.

\section{Asymptotic results for $\mathcal{J}_{1^N; \mathfrak{r}_s}$} \label{Section5}
In this section we prove several asymptotic results about the measure $\mathcal{J}_{1^N; \mathfrak{r}_s}$ from Proposition \ref{propJM}, which were used throughout the text. The key idea, which enables our analysis is that $\mathcal{J}_{1^N; \mathfrak{r}_s}$ can be identified with the discrete $\beta$-ensemble of \cite{BGG}.
\subsection{Discrete $\beta$-ensemble identification}\label{Section5.1}
 We start by giving the definition of the discrete $\beta$-ensemble as in \cite{BGG}.
\begin{definition}\label{S5DefBeta}
Fix $N \in \mathbb{N}$, $2\theta = \beta>0$ and a real-valued function $w(x;N)$.\footnote{ $w(x;N)$ should decay at least as $|x|^{-2\theta( 1 + \epsilon)}$ for some $\epsilon > 0$ as $|x| \rightarrow \infty$.} With the above data we define the discrete $\beta$-ensemble as the probability distribution
\begin{equation}\label{S5beta}
\mathbb{P}_N(\ell_1,...,\ell_N)=\frac{1}{Z_N}\prod_{1\leq{i}<j\leq{N}}\frac{\Gamma(\ell_j-\ell_i+1)\Gamma(\ell_j-\ell_i+\theta)}{\Gamma(\ell_j-\ell_i)\Gamma(\ell_j-\ell_i+1-\theta)}\prod_{i=1}^{N}w(\ell_i;N)
\end{equation}
on ordered $N$-tuples $\ell_1<\cdots <\ell_N$ such that $\displaystyle{\ell_i}=\lambda_{N-i + 1}+{\theta} \cdot i$ and $\lambda_1\geq {...}\geq{\lambda_N}$ are integers. The quantity $Z_N$ is a normalization constant, which is finite under the assumptions on $w(x;N)$. We denote the state space of the above configurations $(\ell_1,...\ell_N)$ by $\mathbb{W}_N^{\theta}$. 
\end{definition}
\begin{remark}\label{S5rem1}
The probability in (\ref{S5beta}) looks like $\prod_{1\leq i<j\leq N}(\ell_j-\ell_i)^{\beta}\prod_{i=1}^{N}w(\ell_i;N)$ if $\ell_j-\ell_i\to\infty$ for $1\leq i < j \leq N$. The latter describes the general-$\beta$ log-gas probability distribution and one can think of the discrete $\beta$-ensemble as a certain discrete version of it.
\end{remark}   
  
If we set $\ell_i=\lambda_{N-i+1}+\theta \cdot i$ with $ ( \lambda_1 \geq \cdots \geq \lambda_N) = \lambda \in \mathbb{Y}^N$, distributed according to $\mathcal{J}_{1^N; \mathfrak{r}_s}$ from Proposition \ref{propJM}, then we obtain
\begin{equation}\label{S5eq2} 
\displaystyle{\mathbb{P}(\ell_1,...,\ell_N)=\frac{\Gamma(\theta)^Ne^{-s\theta N}(s\theta)^{-\frac{N(N+1)}{2}} }{\prod_{i = 1}^N\Gamma(i\theta)}\prod_{1\leq{i}<j\leq{N}}\frac{\Gamma(\ell_j-\ell_i+\theta)\Gamma(\ell_j-\ell_i+1)}{\Gamma(\ell_j-\ell_i)\Gamma(\ell_j-\ell_i+1-\theta)}\prod_{i=1}^{N}\frac{(s\theta)^{\ell_i}}{\Gamma(\ell_i+1)}}
\end{equation}

The above shows that $\mathcal{J}_{1^N; \mathfrak{r}_s}$ is equivalent with the discrete $\beta$-ensemble with $w(x;N)=\frac{(s\theta)^x}{\Gamma(x+1)}$. The latter implies that we may use the results in \cite{BGG} to derive various asymptotic statements about the measure $\mathcal{J}_{1^N; \mathfrak{r}_s}$. We begin with a law of large numbers for the empirical measures.
\begin{theorem}\label{empThm}
Fix $s \geq 0$ and $t > 0$ and let $\lambda \in \mathbb{Y}^N$ be distributed according to $\mathcal{J}_{1^N; \mathfrak{r}_{tN + s}}$ from Proposition \ref{propJM}.
Suppose $\mu^{s,t}_N= \frac{1}{N}\sum_{i = 1}^N \delta \left( \frac{\lambda_i + \theta \cdot (N-i+1)}{N}\right)$. Then there exists a deterministic measure $\mu^{s,t}$, such that $\mu^{s,t}_N \Rightarrow \mu^{s,t}$ as $N \rightarrow \infty$, in the sense that for any bounded continuous function $f$, we have the following convergence in probability:
$$\lim_{N \rightarrow \infty} \int_{\mathbb{R}}{f(x)d\mu^{s,t}_N(x)}={\int_{\mathbb{R}}{f(x)d\mu^{s,t}}(x)}.$$
\end{theorem}
\begin{proof}
The result follows from the identification of $\mathcal{J}_{1^N; \mathfrak{r}_{tN + s}}$ with the discrete $\beta$-ensemble in (\ref{S5eq2}) and Theorem 1.2 in \cite{BGG}.
The idea is to establish a large deviations principle for the measure in (\ref{S5eq2}), which would show that it is concentrated on those $N$-tuples $(\ell_1,...,\ell_N)$ which maximize the probability density. Similar results are known in various contexts (see e.g. the references in the proof of Proposition 2.2 of \cite{BGG}).\\
\end{proof}

The measure $\mu^{s,t}$ will be explicitly computed in Section \ref{Section5.2}; however, we remark that it only depends on $t$ and so we will refer to it as $\mu^t$. As will be shown, $\mu^t$ is compactly supported on the interval $[0, b_t]$ with $b_t = \theta(1+\sqrt{t})^2$ and has a density there that is bounded by $\theta^{-1}$. An important additional result that we will require for our discrete $\beta$-ensemble is that the rescaled rightmost particle $\frac{\ell_N}{N}$ concentrates near the right endpoint of the support $\theta(1+\sqrt{t})^2$. We summarize the result in a theorem below, which was communicated to us by Vadim Gorin, and whose proof will appear at a later time.
\begin{theorem}\label{thmRight}
Fix $s \geq 0$ and $t > 0$ and let $\lambda \in \mathbb{Y}^N$ be distributed according to $\mathcal{J}_{1^N; \mathfrak{r}_{tN + s}}$ from Proposition \ref{propJM}. Then we have the following $L^1$ convergence result
\begin{equation}\label{S5eqRight}
\lim_{N \rightarrow \infty} \mathbb{E} \left[ \left| \frac{\lambda_1 + \theta \cdot N }{N} - \theta(1+\sqrt{t})^2\right| \right] = 0.
\end{equation}
\end{theorem}
\begin{remark}
Theorem \ref{thmRight} was proved in the case $\theta = 1$  in \cite{Joh00}, and analogues for continuous log-gases are well-known (see e.g. Section 2.6.2 of \cite{AGZ}). For the general discrete $\beta$-ensemble there is a proof of Theorem \ref{thmRight} under stronger assumptions on the measure in \cite{BGG}, and a more general version which will contain the above theorem as a special case will appear in \cite{GGG}.
\end{remark}

\subsection{Nekrasov's equation}\label{Section5.2}
In this section we use Nekrasov's equation to find the limiting equilibrium measure $\mu^{s,t}$ of Theorem \ref{empThm}. This approach was followed in \cite{BGG}. In the end we will prove Proposition \ref{empProp} and the two properties after it.

The following corollary contains the Nekrasov's equation and can be proved in the same way as Theorem 4.1 in \cite{BGG}. We remark that while we do not have the compactness assumption from that theorem, the same proof goes through, because the set ${\mathbb{W}_N^{\theta}}$ is discrete in $\mathbb{R}$.
\begin{corollary}\label{NekThm}
Let $\mathbb{P}_N$ be the distribution on $N$-tuples $(\ell_1,...,\ell_N)\in{\mathbb{W}_N^{\theta}}$ as in (\ref{S5eq2}). Define 
$$R_N(\xi)=\xi\cdot \mathbb{E}_{\mathbb{P}_N}\left[\prod_{i=1}^N\left(1-\frac{\theta}{\xi-\ell_i}\right)\right]+ (tN + s)\theta  \cdot \mathbb{E}_{\mathbb{P}_N}\left[\prod_{i=1}^{N}\left(1+\frac{\theta}{\xi-\ell_i-1} \right)\right].$$
Then $R_N(\xi)$ is a degree one polynomial.
\end{corollary}
For a probability measure $\nu$ on $\mathbb{R}$ we define the Stieltjes transform 
$$G_\nu(z)=\int_{\mathbb{R}} \frac{1}{z-x}d\nu(x)$$ 
for $z \not \in $ supp$(\nu)$. Note that $G_\nu(z)$ is analytic on the upper and lower complex half-planes. 

We go back to the setup of Theorem \ref{empThm} and assume that $\lambda \in \mathbb{Y}^N$ is distributed according to $\mathcal{J}_{1^N; \mathfrak{r}_{tN + s}}$. Setting  $\ell_i=\lambda_{N-i+1}+\theta \cdot i$, using Corollary \ref{NekThm} and then setting $\xi = Nz$ we conclude that
\begin{equation}\label{S5An1}
R_N(Nz) = Nz\cdot \mathbb{E}_{\mathbb{P}_N}\left[\prod_{i=1}^N\left(1-\frac{\theta}{Nz-\ell_i}\right)\right]+(tN + s)\cdot \mathbb{E}_{\mathbb{P}_N}\left[\prod_{i=1}^{N}\left(1+\frac{\theta}{Nz-\ell_i-1} \right)\right]
\end{equation}
is a degree $1$ polynomial of $z$. Using the approximation $1+x\approx{e^x}$ for small $x$ we see that for $z \in \mathbb{C} \backslash \mathbb{R}$
$$\prod_{i=1}^N\left(1-\frac{\theta}{Nz-\ell_i}\right) = \exp\left(\frac{\theta}{N} \sum_{i = 1}^N \frac{1}{z-\ell_i/N} + O(N^{-1})\right).$$
From Theorem \ref{empThm}, we know that 
$$\frac{1}{N} \sum_{i = 1}^N \frac{1}{z-\ell_i/N} + O(N^{-1}) \xrightarrow{\mathbb{P}} \int_{\mathbb{R}} \frac{1}{z-x}d\mu^t(x) = G_{\mu^{s,t}}(z) , \mbox{ as $N\rightarrow \infty$} $$
where $\mu^{s,t}$ is the limiting measure afforded by the theorem. An application of the Bounded Convergence Theorem shows that
\begin{equation}\label{S5An2}
\lim_{N \rightarrow \infty} \mathbb{E}_{\mathbb{P}_N}\left[\prod_{i=1}^N\left(1-\frac{\theta}{Nz-\ell_i}\right)\right] = \exp(\theta G_{\mu^{s,t}}(z)).
\end{equation}
Similar arguments reveal that
\begin{equation}\label{S5An3}
\lim_{N \rightarrow \infty} \mathbb{E}_{\mathbb{P}_N}\left[\prod_{i=1}^{N}\left(1+\frac{\theta}{Nz-\ell_i-1} \right)\right] = \exp(-\theta G_{\mu^{s,t}}(z)).
\end{equation}
Dividing both sides of (\ref{S5An1}) by $N$ and letting $N$ tend to infinity we conclude that for each $z \in \mathbb{C} \backslash \mathbb{R}$, we have
$$\lim_{N \rightarrow \infty} \frac{R_N(Nz)}{N} = z \exp(\theta G_{\mu^{s,t}}(z)) + t\exp(-\theta G_{\mu^{s,t}}(z)).$$

Since $R_N$ is a degree $1$ polynomial, we know that $N^{-1}R_N(Nz) = a_Nz + b_N$ for some sequences $a_N, b_N \in \mathbb{C}$. The above equation suggests that $a_N \rightarrow a$ and $b_N \rightarrow b$ as $N \rightarrow \infty$ for some $a,b \in \mathbb{C}$. In addition, using that $G_{\mu^{s,t}}(\iota x) \sim \frac{1}{\iota x}$ as $x \rightarrow \infty$, we conclude that $a = 1$ and $b = \theta(t-1)$. The latter means that we have the following functional equation for the Stieltjes transform of the limiting measure $\mu^{s,t}$
\begin{equation}\label{S5An4}
z + \theta(t-1) =  z \exp(\theta G_{\mu^{s,t}}(z)) + t\exp(-\theta G_{\mu^{s,t}}(z)).
\end{equation}

We observe that (\ref{S5An4}) is a quadratic equation in $\exp(\theta G_{\mu^{s,t}}(z))$ and we can solve it to get
\begin{equation}\label{S5An5}
\exp(\theta G_{\mu^{s,t}}(z))=\frac{z+\theta(t-1)-\sqrt{(z+\theta(t-1))^2-4t\theta z}}{2t\theta}
\end{equation}
We take logarithms above and invoke the Stieltjes transform inversion formula (see e.g. Theorem 2.4.3 in \cite{AGZ})
$$\displaystyle{f(x)=\lim_{y\to0^+}\frac{ImG_{\mu^{s,t}}(x-iy)-ImG_{\mu^{s,t}}(x+iy)}{2\pi i}}$$
to derive a formula for the density of the limitng measure. The result is presented below and we split it into the cases $t \geq 1$ and $t \in (0,1)$.

Suppose $t\geq 1$. Then we get 
\begin{equation}\label{S5Den1}
f(x) = \begin{cases} 0 &\mbox{  for $x<\theta(\sqrt{t}-1)^2$ or $x>\theta(\sqrt{t}+1)^2$,} \\ 
(\theta \pi)^{-1}\arccot\left(\frac{x+\theta(t-1)}{\sqrt{4{\theta}tx-[x+\theta(t-1)]^2}}\right) &\mbox{ otherwise.}
\end{cases}
\end{equation}
Suppose $t \in (0,1)$. Then we get 
\begin{equation}\label{S5Den2}
f(x) = \begin{cases} 0 &\mbox{ for $x>\theta(\sqrt{t}+1)^2$}, \\ \theta^{-1} &\mbox{ for $x<\theta(\sqrt{t}-1)^2$},\\
(\theta \pi)^{-1}\arccot\left(\frac{x+\theta(t-1)}{\sqrt{4{\theta}tx-[x+\theta(t-1)]^2}}\right) &\mbox{ otherwise.}
\end{cases}
\end{equation}

We end the section with a proof of  Proposition \ref{conProp} and the two properties after it.
\begin{proof}(Proposition \ref{empProp})
By Proposition \ref{S2PC} and Lemma \ref{S2LMJ} we know that the distribution of $X^N(tN+s;N)$ is the same as $\mathcal{J}_{1^N; \mathfrak{r}_{tN + s}}$. Consequently, the convergence statement of the proposition follows from Theorem \ref{empThm}. The fact that the limit depends only on $t$ and not $s$ is a consequence of (\ref{S5Den1}) and (\ref{S5Den2}), which also imply the first property after Proposition \ref{empProp}.

From (\ref{S5Den1}) and (\ref{S5Den2}) we see that the density $f$ behaves like ${\bf 1}_{\{x < b_t\}}\cdot \sqrt{b_t - x}$ near $b_t = \theta(\sqrt{t}+1)^2$. Thus the second property after Proposition   \ref{empProp} is a consequence of the Dominated Convergence Theorem. 
\end{proof}

\subsection{Proof of Proposition \ref{conProp} and Lemma \ref{S4Lemma1}}\label{Section5.3}
We begin with a useful lemma.
\begin{lemma}\label{S5Lemma}Let $\mathbb{P}^s_N$ be the distribution on $N$-tuples $(\ell_1,...,\ell_N)\in{\mathbb{W}_N^{\theta}}$ as in (\ref{S5eq2}) with $s > 0$. Define
$$\Delta(x)=\prod_{j=1}^{N-1}\frac{\Gamma(x-\ell_j+1)\Gamma(x-\ell_j+\theta)}{\Gamma(x-\ell_j+1-\theta)\Gamma(x-\ell_j)}.$$ 
Then we have
\begin{equation}\label{S5M1}
\mathbb{E}_{\mathbb{P}^s_N}\left[\frac{\Delta(\ell_{N}-1)}{\Delta(\ell_N)}\right] \hspace{-1mm}= \mathbb{E}_{\mathbb{P}^s_N}\left[\frac{s\theta}{\ell_N+1}\right] \mbox{ and } \mathbb{E}_{\mathbb{P}^s_N}\left [\frac{\Delta(\ell_N+1)}{\Delta(\ell_N)}\right]  \hspace{-1mm} = \mathbb{E}_{\mathbb{P}^s_N}\left[\frac{\ell_N}{s\theta} \times {\bf 1}_{\{\ell_N > \ell_{N-1}+\theta\}} \right].
\end{equation}
\end{lemma}
\begin{proof}
Using the functional equation $\Gamma(x+1)=x\Gamma(x)$, we see that
$$\mathbb{E}_{\mathbb{P}^s_N}\left[\frac{\ell_N}{s\theta}{\Bigg |}\ell_1,...,\ell_{N-1}\right] = C(\ell_1,..,\ell_{N-1})\sum_{l_N\in{\theta+\ell_{N-1}+ \mathbb{Z}_{\geq 0}}}\frac{\ell_N}{s\theta}\frac{(s\theta)^{\ell_N}}{\Gamma(\ell_N+1)}\Delta(\ell_N)= $$
$$ =C(\ell_1,..,\ell_{N-1})\left[ \frac{(s\theta)^{\ell_{N-1} + \theta-1}}{\Gamma(\ell_{N-1} + \theta)}\Delta(\ell_{N-1} + \theta) + \sum_{\ell_N\in{\theta+\ell_{N-1}+\mathbb{Z}_{\geq 0}}}\frac{(s\theta)^{\ell_N}}{\Gamma(\ell_N+1)}\Delta(\ell_N)\frac{\Delta(\ell_N+1)}{\Delta(\ell_N)}   \right]=$$
$$ =  \mathbb{E}_{\mathbb{P}^s_N}\left [\frac{\ell_N}{s\theta} \times {\bf 1}_{\{\ell_N = \ell_{N-1}+\theta\}} {\Bigg |}\ell_1,...,\ell_{N-1} \right ] + \mathbb{E}_{\mathbb{P}^s_N}\left [\frac{\Delta(\ell_N+1)}{\Delta(\ell_N)}{\Bigg |}\ell_1,...,\ell_{N-1} \right ],$$
where $C(\ell_1,..,\ell_{N-1})$ is some normalization constant. Rearranging terms and taking the expectation on both sides above we conclude the second part of the lemma.

The first part is proved similarly.
$$\mathbb{E}_{\mathbb{P}^s_N}\left[\frac{s\theta}{\ell_N+1}{\Bigg |}\ell_1,...,\ell_{N-1}\right] = C(\ell_1,..,\ell_{N-1})\sum_{\ell_N\in{\theta+\ell_{N-1}+ \mathbb{Z}_{\geq 0}}}\frac{s\theta}{\ell_N+1}\frac{(s\theta)^{\ell_N}}{\Gamma(\ell_N+1)}\Delta(\ell_N)= $$
$$ =C(\ell_1,..,\ell_{N-1}) \sum_{l_N\in{\theta+\ell_{N-1}+\mathbb{Z}_{\geq 0}}}\frac{(s\theta)^{\ell_N}}{\Gamma(\ell_N+1)}\Delta(\ell_N)\frac{\Delta(\ell_N-1)}{\Delta(\ell_N)} = \mathbb{E}_{\mathbb{P}^s_N}\left [\frac{\Delta(\ell_N-1)}{\Delta(\ell_N)}{\Bigg |}\ell_1,...,\ell_{N-1}\right ],$$
where $C(\ell_1,..,\ell_{N-1})$ is some normalization constant. In the above we used that $\Delta(\ell_{N-1} + \theta - 1) = 0$. Taking expectations on both sides of the above equation proves the first statement in the lemma.
\end{proof}

\begin{proof}(Proposition \ref{conProp})
From Proposition \ref{S2PC} and Lemma \ref{S2LMJ} we know that the distribution of $X^N(s;N)$ is the same as $\mathcal{J}_{1^N; \mathfrak{r}_{s}}$. The latter together with Theorem \ref{thmRight} proves that
$$\frac{X_1^N(tN+s;N) + \theta \cdot N}{N}  \xrightarrow{L^1} \theta(1 + \sqrt{t})^2\mbox{ as $N \rightarrow \infty$},$$
which is equivalent to (\ref{S3Right}).

For $\lambda$ distributed according to $\mathcal{J}_{1^N; \mathfrak{r}_{s}}$, we set $\ell_i=\lambda_{N-i+1}+\theta \cdot i$. Recall that $(\ell_1,...,\ell_N)$ has the same distribution as (\ref{S5eq2}). Let $\Delta(x)$ be as in Lemma \ref{S5Lemma} and notice that from the functional equation of the gamma function $\Gamma(x+1) = x\Gamma(x)$, we have
$$\frac{\Delta(\ell_{N}-1)}{\Delta(\ell_N)} = \prod_{j = 1}^{N-1} \left( 1 - \frac{1}{\ell_N - \ell_j} \right) \left( 1 - \frac{2\theta - 1}{\ell_N - \ell_j + \theta - 1}\right).$$
Combining the above with the first equation in (\ref{S5M1}), where we replace $s$ with $tN + s$, we conclude
\begin{equation}\label{S5M3}
\mathbb{E}_{\mathbb{P}^{tN + s}_N}\left[\prod_{j = 1}^{N-1} \left( 1 - \frac{1}{\ell_N - \ell_j} \right) \left( 1 - \frac{2\theta - 1}{\ell_N - \ell_j + \theta - 1}\right)\right]  = \mathbb{E}_{\mathbb{P}^{tN + s}_N}\left[\frac{(tN + s)\theta}{\ell_N + 1}\right].
\end{equation}
Observe that $\frac{\ell_N + 1}{(tN + s)\theta} \xrightarrow{L^1} t^{-1} (1 + \sqrt{t})^2$ from Theorem \ref{thmRight}, and so $\frac{(tN + s)\theta}{\ell_N + 1} \xrightarrow{\mathbb{P}} t (1 + \sqrt{t})^{-2}$. In addition, we know that $\ell_N \geq N \theta$ by definition and so the Bounded Convergence Theorem shows
\begin{equation}\label{S5M4}
\lim_{N \rightarrow \infty} \mathbb{E}_{\mathbb{P}^{tN + s}_N}\left[\frac{(tN + s)\theta}{\ell_N + 1}\right] = \frac{t}{(1 + \sqrt{t})^2}.
\end{equation}
Equations (\ref{S5M3}) and (\ref{S5M4}) prove (\ref{ObsF}) and hence Proposition \ref{conProp}.
\end{proof}
\begin{proof}(Lemma \ref{S4Lemma1})
From Proposition \ref{S2PC} and Lemma \ref{S2LMJ} we know that the distribution of $X^N(s;N)$ is the same as $\mathcal{J}_{1^N; \mathfrak{r}_{s}}$. For $\lambda$ distributed according to $\mathcal{J}_{1^N; \mathfrak{r}_{s}}$, we set $\ell_i=\lambda_{N-i+1}+\theta \cdot i$. Recall that $(\ell_1,...,\ell_N)$ has the same distribution as (\ref{S5eq2}). Let $\Delta(x)$ be as in Lemma \ref{S5Lemma} and notice that from the functional equation of the gamma function $\Gamma(x+1) = x\Gamma(x)$, we have
$$\frac{\Delta(\ell_{N}+1)}{\Delta(\ell_N)} = \prod_{j = 1}^{N-1} \left( 1 + \frac{1}{\ell_N - \ell_j} \right) \left( 1 + \frac{2\theta - 1}{\ell_N - \ell_j + \theta - 1}\right).$$
Combining the above with the second equation in (\ref{S5M1}) we conclude
\begin{equation}\label{S5M5}
\mathbb{E}_{\mathbb{P}^{ s}_N}\left[\prod_{j = 1}^{N-1} \left( 1 + \frac{1}{\ell_N - \ell_j} \right) \left( 1 + \frac{2\theta - 1}{\ell_N - \ell_j + \theta - 1}\right)\right]  = \mathbb{E}_{\mathbb{P}^{s}_N}\left[\frac{\ell_N }{s\theta} \times {\bf 1}_{\{\ell_N > \ell_{N-1}+\theta\}} \right] .
\end{equation}
Using that $\ell_N - \ell_i \geq \theta \geq 1$, we see that
\begin{equation}\label{S5M6}
\prod_{j = 1}^{N-1} \left( 1 + \frac{\theta}{\ell_N - \ell_j} \right) \leq \prod_{j = 1}^{N-1} \left( 1 + \frac{1}{\ell_N - \ell_j} \right) \left( 1 + \frac{2\theta - 1}{\ell_N - \ell_j + \theta - 1}\right).
\end{equation}
Combining (\ref{S5M5}), (\ref{S5M6}) with the inequality $\mathbb{E}_{\mathbb{P}^{s}_N}\left[\frac{\ell_N }{s\theta} \times {\bf 1}_{\{\ell_N > \ell_{N-1}+\theta\}} \right] \leq \mathbb{E}_{\mathbb{P}^{s}_N}\left[\frac{\ell_N }{s\theta} \right]$, we conclude (\ref{S4ObsF}).\\

In what follows we will prove (\ref{S4ObsF2}). For $\lambda$ distributed according to $\mathcal{J}_{1^N; \mathfrak{r}_{tN + s}}$, we set $\ell_i=\lambda_{N-i+1}+\theta \cdot i$. Since we already proved Propositions \ref{empProp} and \ref{conProp} we may use the results from Lemma \ref{LSep}. They imply that 
\begin{equation}\label{S5M7}
\prod_{j = 1}^{N-1} \left( 1 + \frac{1}{\ell_N - \ell_j} \right) \left( 1 + \frac{2\theta - 1}{\ell_N - \ell_j + \theta - 1}\right) \xrightarrow{\mathbb{P}} \frac{(\sqrt{t} + 1)^2}{t}\mbox{ as $N \rightarrow \infty$}.
\end{equation}
In addition, by  Lemma \ref{LSep} we know that ${\bf 1}_{\{\ell_N > \ell_{N-1}+\theta\}} \xrightarrow{\mathbb{P}} 1\mbox{ as $N \rightarrow \infty$}$ and so Theorem \ref{thmRight} together with the Generalized Dominated Convergence Theorem implies that
\begin{equation}\label{S5M8}
\lim_{N \rightarrow \infty} \mathbb{E}_{\mathbb{P}^{tN + s}_N}\left[\frac{\ell_N }{s\theta} \times {\bf 1}_{\{\ell_N > \ell_{N-1}+\theta\}} \right] =  \frac{(\sqrt{t} + 1)^2}{t}.
\end{equation}
Combining (\ref{S5M5}), (\ref{S5M7}) and (\ref{S5M8}) we conclude that
\begin{equation}\label{S5M9}
\prod_{j = 1}^{N-1} \left( 1 + \frac{1}{\ell_N - \ell_j} \right) \left( 1 + \frac{2\theta - 1}{\ell_N - \ell_j + \theta - 1}\right) \xrightarrow{L^1} \frac{(\sqrt{t} + 1)^2}{t} \mbox{ as $N \rightarrow \infty$}.
\end{equation}

From Lemma \ref{LSep} we know that $\prod_{j = 1}^{N-1} \left( 1 +  \frac{\theta}{\ell_N - \ell_j} \right) \xrightarrow{\mathbb{P}} \frac{\sqrt{t} + 1}{\sqrt{t}}\mbox{ as $N \rightarrow \infty$}$. The latter, together with (\ref{S5M6}), (\ref{S5M9}) and the Generalized Dominated Convergence Theorem implies that
\begin{equation}\label{S5M10}
\prod_{j = 1}^{N-1} \left( 1 +  \frac{\theta}{\ell_N - \ell_j} \right)\xrightarrow{L^1} \frac{\sqrt{t} + 1}{\sqrt{t}} \mbox{ as $N \rightarrow \infty$}.
\end{equation}
Equation (\ref{S5M10}) implies (\ref{S4ObsF2}).
\end{proof}

\bibliographystyle{amsplain}
\bibliography{PD}

\end{document}